\documentclass{article}
\usepackage{bm}
\usepackage{tikz,mathpazo}
\usepackage{pgf}
\usepackage[top=2.5cm, bottom=2.5cm, left=3cm, right=3cm]{geometry}   
\usepackage{indentfirst}
\usepackage{graphicx}
\usepackage{graphics}
\usepackage[toc,page,title,titletoc,header]{appendix}
\usepackage{bm}
\usepackage{listings}  
\usepackage{amsmath}
\usepackage{setspace} 
\usepackage{indentfirst}
\usepackage{caption}
\usepackage{multirow} 
\usepackage{lipsum,multicol}
\usepackage{pdfpages}
\usepackage{float}
\usepackage{amsthm}
\usepackage{pdfpages}
\usepackage{url}
\usepackage{colortbl}
\usepackage{subfigure}
\usepackage{epsfig}
\usepackage{epstopdf}
\usetikzlibrary{shapes.geometric, arrows}
\usepackage{fancyhdr}
\usepackage{abstract}
\usepackage{tikz,mathpazo}
\usetikzlibrary{shapes.geometric, arrows}

\usepackage{amssymb}
\usepackage{latexsym}
\usepackage{verbatim}
\usepackage[numbers]{natbib} 
\usepackage{booktabs}
\usepackage{tikz-cd}

\allowdisplaybreaks[4]

\newcommand{\inner}[1]{\left\langle #1 \right\rangle}
\newcommand{\norm}[1]{\left\Vert #1\right\Vert}
\newcommand{\bb}[1]{\mathbb{#1}}

\newcommand{\X}{{ \ca{X} }}

\newcommand{\ca}[1]{\mathcal{#1}}
\newcommand{\tr}[0]{\mathrm{tr}}

\newcommand{\Diag}[0]{\mathrm{Diag}}

\newcommand{\M}[0]{\mathcal{M}}

\newcommand{\tp}{^\top}

\newcommand{\A}{\ca{A}}

\newcommand{\Tx}{{\ca{T}_x}}
\newcommand{\Nx}{\ca{N}_x}

\newcommand{\xk}{{x_{k} }}

\newcommand{\Jc}{{{J}_c}}
\newcommand{\Ja}{{{J}_{A}}}
\newcommand{\Ju}{{{J}_u}}
\newcommand{\Jv}{{{J}_{v}}}

\newcommand{\Lf}{ {M_{x,f}} }
\newcommand{\Mu}{ {M_{x,u}} }
\newcommand{\Mv}{ {M_{x,v}} }

\newcommand{\Lsc}{ \sigma_{x,c} }

\newcommand{\Mc}{ {M_{x,c}} }

\newcommand{\Ma}{ M_{x,A} }

\newcommand{\Lc}{ {L_{x,c}} }

\newcommand{\La}{ {L_{x,A}} }

\newcommand{\Lac}{ {L_{x, b}} }
\newcommand{\Mxlambdamu}{M_{x,\lambda,\mu}}

\newcommand{\Omegax}[1]{ {\Omega_{#1}} }

\newcommand{\BOmegax}[1]{ {\overline{\Omega}_{#1}} }

\newcommand{\y}{{y}}

\newcommand{\K}{\ca{K}}

\newcommand{\KA}{\ca{K}_{A}}

\newcommand{\ic}{v}
\newcommand{\ec}{u}

\newcommand{\TK}{\ca{T}_{\K}}
\newcommand{\TKlin}{{\ca{T}_{\K}^{lin}}}
\newcommand{\TKA}{\ca{T}_{\KA}}
\newcommand{\TKAlin}{{\ca{T}_{\KA}^{lin}}}

\newcommand{\Rn}{\mathbb{R}^n}
\newcommand{\Rp}{\mathbb{R}^p}

\newcommand{\LNLP}{\ca{L}_{NLP}}
\newcommand{\LCDP}{\ca{L}_{CDP}}

\newtheorem{theo}{Theorem}[section]
\newtheorem{lem}[theo]{Lemma}
\newtheorem{prop}[theo]{Proposition}
\newtheorem{examples}[theo]{Example}
\newtheorem{cond}[theo]{Condition}
\newtheorem{coro}[theo]{Corollary}

\newtheorem{defin}[theo]{Definition}
\newtheorem{rmk}[theo]{Remark}
\newtheorem{assumpt}[theo]{Assumption}

\usepackage[figuresright]{rotating}
\usepackage{pdflscape}
\usepackage{url}

\usepackage{caption}
\usepackage{algorithm}
\usepackage{algpseudocode}
\usepackage{longtable}
\usepackage{appendix}

\numberwithin{equation}{section}

\title{A Partial Exact Penalty Function Approach for Constrained Optimization}
\author{Nachuan Xiao
	\thanks{The Institute of Operations Research and Analytics, National University of Singapore, Singapore. (xnc@lsec.cc.ac.cn). The research of this author is supported by the Ministry of Education, Singapore, under its Academic Research Fund Tier 3 grant call (MOE-2019-T3-1-010).} ~
	Xin Liu\thanks{State Key Laboratory of Scientific and Engineering Computing, Academy of Mathematics and Systems Science, Chinese Academy of Sciences, and University of Chinese Academy of Sciences, China (liuxin@lsec.cc.ac.cn). Research is supported in part by the National Natural Science Foundation of China (No. 12125108, 11971466, 11991021), Key Research Program of Frontier Sciences, Chinese Academy of Sciences (No. ZDBS-LY-7022).} ~
	and Kim-Chuan Toh \thanks{Department of Mathematics, and Institute of Operations Research and Analytics, National University of Singapore, Singapore 119076 (mattohkc@nus.edu.sg). The research of this author is supported by the Ministry of Education, Singapore, under its Academic Research Fund Tier 3 grant call (MOE-2019-T3-1-010).}}

\begin{document}
	\maketitle
	
	\begin{abstract}
		
		In this paper, we focus on a class of constrained nonlinear optimization problems (\ref{Prob_Ori}), where some of its equality constraints define a closed embedded submanifold $\M$ in $\mathbb{R}^n$. Although \ref{Prob_Ori} can be solved directly by various existing approaches for constrained optimization in Euclidean space, these approaches usually fail to recognize the manifold structure of $\M$. 
		To achieve better efficiency by utilizing the manifold structure of $\M$ in directly applying these existing optimization approaches,  we propose a partial penalty function approach for \ref{Prob_Ori}. In our proposed penalty function approach, we transform \ref{Prob_Ori}  into  the corresponding constraint dissolving problem (CDP) in the Euclidean space, where the constraints that define $\M$ are eliminated through exact penalization. We establish the relationships on the constraint qualifications between \ref{Prob_Ori} and CDP, and prove that \ref{Prob_Ori} and CDP have the same stationary points and KKT points in a neighborhood of the feasible region under mild conditions. 
		Therefore, various existing optimization approaches developed for constrained optimization in the Euclidean space can be directly applied to solve \ref{Prob_Ori} through CDP. Preliminary numerical experiments demonstrate that by dissolving the constraints that define $\M$, CDP gains superior computational efficiency when compared to directly applying existing optimization approaches to solve \ref{Prob_Ori}, especially in high dimensional scenarios.
		
	\end{abstract}

	\section{Introduction}
	In this paper, we consider the following constrained optimization problem,
	\begin{equation}
		\label{Prob_nlp}
		\begin{aligned}
			\min_{x \in \bb{R}^n} \quad & f(x)\\
			\text{s.t.} \quad & s(x) = 0,~ \ic(x) \leq 0, 
		\end{aligned}
	\end{equation}
	where a subset of the equality constraints $s(x) = 0$ satisfies the {\it linear independence constraint qualification} (LICQ). Therefore, we can reshape \eqref{Prob_nlp} as the following constrained nonlinear optimization problem (\ref{Prob_Ori}),
	\begin{equation}
		\label{Prob_Ori}
		\tag{NLP}
		\begin{aligned}
			\min_{x \in \bb{R}^n} \quad & f(x)\\
			\text{s.t.} \quad & x \in \M := \{x \in \Rn: c(x) =0\},\\
			& \ec_i(x) = 0, \quad \forall i \in [N_E] :=\{ 1,..., N_E \},\\
			& \ic_j(x) \leq 0, \quad \forall j \in [N_I]:=\{ 1,..., N_I \}.  
		\end{aligned}
	\end{equation}
	where $c(x) = 0$ refers to the subset of the equality constraints $s(x) = 0$ in \eqref{Prob_nlp} that satisfies LICQ, and $u(x) = 0$ refers to all the remaining equality constraints. Throughout this paper,  we make the following assumptions on \ref{Prob_Ori}, 
	\begin{assumpt}{\bf Default assumptions}
		\label{Assumption_1}
		
		\begin{enumerate} 
			\item $f: \bb{R}^n \to \bb{R}$  is locally Lipschitz continuous and Clarke regular in $\bb{R}^{n}$. The definition of Clarke regular functions can be found in \cite[Definition 2.3.4]{clarke1990optimization} or Definition \ref{Defin_Clarke_regular} in this paper.
			\item $c: \bb{R}^n \to \bb{R}^p$ is a locally Lipschitz smooth mapping, i.e., the transposed Jacobian of $c$, denoted as $\Jc(x)$, is locally Lipschitz continuous over $\bb{R}^n$. 
			\item The constraints $c(x) = 0$ satisfies LICQ over $\M$. That is, $\Jc(x)$ is full-rank for any $x \in \M$. 
			\item $u: \bb{R}^{n}\to \bb{R}^{N_E}$ and $v: \bb{R}^{n}\to \bb{R}^{N_I}$ are locally Lipschitz smooth mappings.
			\item The feasible region of \ref{Prob_Ori}, denoted as $\K:= \{ x\in \M: u(x) = 0, ~ v(x) \leq 0 \}$, is nonempty.
		\end{enumerate}
	\end{assumpt}
	 
	Optimization problems that take the form of \ref{Prob_Ori} have wide applications in data science and many other related areas, including computing the geometric mean over a Riemannian manifold \cite{karcher1977riemannian,moakher2002means,berger2003panoramic,afsari2011riemannian,bergmann2019intrinsic}, maximum balanced cut problems \cite{liu2019simple}, quadratic assignment problems \cite{burkard1997qaplib}, clustering problems \cite{wang2021clustering}, etc. In the following, we briefly present  some examples of \ref{Prob_Ori}.

	\begin{examples}[Riemannian center of mass problem]
		\label{Example_Rie_mass}
		Riemannian center of mass problem has important applications in pure mathematics \cite{karcher1977riemannian,berger2003panoramic,afsari2011riemannian}, data science \cite{moakher2002means,bergmann2019intrinsic}, etc. 
		Given the Riemannian manifold $\M := \{x \in \Rn: c(x) = 0\}$ and the samples $\{s_i: i\in [N]\}$ from $\M$, the Riemannian center of mass can be expressed as
		\begin{equation}
			\begin{aligned}
				\min_{x \in \bb{R}^n}\quad &\frac{1}{N} \sum_{i = 1}^N \norm{x - s_i}^2\\
				\text{s.t.} \quad & c(x) = 0, ~\norm{x - s^*}^2 \leq r,
			\end{aligned}
		\end{equation}
		where $s^*$ is a prefixed point on $\M$ and $r >0$ is a constant.

	\end{examples}
	
	\begin{examples}[Minimum balanced cut for graph bisection]
		\label{Example_Graph_cut}
		Given an undirected graph, the minimum cut problem aims to separate the vertices into two clusters that have the same size, while enforcing that the number of edges between the two clusters is as small as possible. As described in \cite{liu2019simple}, such a problem can be relaxed into the following Riemannian optimization problem with additional linear constraints,
		\begin{equation}
			\label{Example_BC}
			\begin{aligned}
				\min_{X \in \bb{R}^{m\times q}} \quad & -\frac{1}{4} \tr\left( X\tp LX \right)\\
				\text{s.t.}\quad & \mathrm{Diag}(XX\tp ) = I_m,\\
				& X\tp e = 0,
			\end{aligned}
		\end{equation}
		where $L$ is the Laplacian matrix of the graph, $e \in \bb{R}^m$ is the vector of $1$'s, and the constraints $\mathrm{Diag}(XX\tp ) = I_m$ defines an Oblique manifold in $\bb{R}^{m\times q}$.
		

	\end{examples}
	
	 As mentioned in \cite{Absil2009optimization}, the equality constraints $c(x) = 0$ in \ref{Prob_Ori} define an embedded submanifold $\M$ in $\Rn$. Therefore, \ref{Prob_Ori} can be regarded as a optimization problem in $\Rn$ with equality constraints $[c(x); u(x)] = 0$ and inequality constraints $v(x) \leq 0$, or be regarded as a optimization problem over the embedded submanifold $\M$ with additional constraints $u(x) = 0$ and $v(x) \leq 0$. 
	
	According to how the constraints $c(x) = 0$ are treated, there are two parallel categories of optimization approaches for solving \ref{Prob_Ori}. One category regards \ref{Prob_Ori} as a standard nonlinear constrained optimization problem in $\Rn$, where the constraints $c(x) = 0$ are nested as an additional set of equality constraints, without any special treatment of its underlying manifold structures. Then a great number of {\it Euclidean optimization approaches} (i.e. approaches designed for constrained optimization problems in $\bb{R}^n$) can  directly be applied for solving \ref{Prob_Ori}. These approaches include augmented Lagrangian method (ALM) \cite{hestenes1969multiplier,powell1969method}, sequential quadratic programming method (SQP) \cite{curtis2017bfgs}, interior-point method \cite{bian2015complexity}, filter method \cite{fletcher2002nonlinear}, etc.  Furthermore, benefited from the rich expertise gained over the past decades for solving constrained optimization in $\bb{R}^n$, various efficient Euclidean solvers (i.e. solvers developed for constrained optimization in $\bb{R}^n$), including  ALGENCAN \cite{andreani2008augmented,andreani2008augmented2} and Ipopt \cite{wachter2006implementation}, are developed and widely applied in solving \ref{Prob_Ori}.

	However, as the constraints $c(x) = 0$ define an embedded submanifold in $\Rn$, another category of optimization approaches \cite{liu2019simple,bergmann2019intrinsic} regard $u(x) = 0$ and $v(x) \leq 0$ as additional constraints over the embedded submanifold $\M$, i.e.,
	\begin{equation}
		\label{Prob_Rie}
		\begin{aligned}
			\min_{x \in \M} \quad & f(x)\\
			\text{s.t.} \quad & \ec(x) = 0,~ \ic(x) \leq 0. 
		\end{aligned}
	\end{equation}
	Following the well-recognized framework presented in \cite{Absil2009optimization}, existing  {\it Riemannian optimization approaches} (i.e. optimization approaches designed for optimization over Riemannian manifolds) are developed by extending efficient Euclidean optimization approaches from $\bb{R}^n$ to $\M$, including Riemannian SQP method \cite{bai2018analysis,schiela2020sqp}, Riemannian augmented Lagrangian method \cite{liu2019simple,tang2021solving}, Riemannian nonsmooth penalty method \cite{liu2019simple}, Riemannian Frank-Wolfe method \cite{weber2019nonconvex}, etc. 
	Compared to directly applying Euclidean optimization approaches to solve \ref{Prob_Ori}, these Riemannian optimization approaches are potentially more efficient by identifying $c(x) = 0$ as a Riemannian manifold  and exploit the geometrical structures of $\M$ \cite{liu2019simple}. 
	
	Although these Riemannian optimization approaches are powerful alternatives for solving \eqref{Prob_Rie}, developing these Riemannian approaches relies on geometrical materials for the underlying manifold, which include Riemannian derivatives, retractions and their inverse, vector transports, etc  \cite{Absil2009optimization,boumal2020introduction}. 
	Determining those geometrical materials is usually challenging, see \cite{Absil2009optimization,boumal2020introduction,son2021symplectic} for instances. As a result, those geometrical materials are well understood only for several well-known Riemannian manifolds, which limits the applications of these existing Riemannian optimization approaches only to several well-known manifolds.  
	
	Moreover, even if those geometrical materials for $\M$ are available,  developing a Riemannian optimization approach from an Euclidean optimization approach by the framework described in \cite{Absil2009optimization} is still not an easy task. To transfer an Euclidean optimization approach to its Riemannian versions, one need to replace the derivatives of $f$, $u$ and $v$ by their Riemannian derivatives, introduce retractions to keep the iterates in $\M$, and employ vector transports to move vectors among different tangent spaces of $\M$.   As a result,  existing Riemannian optimization approaches are  limited, especially when compared to the increasing number of Euclidean optimization approaches. Furthermore, in contrast to various available efficient {\it Euclidean solvers} (i.e., solvers developed based on Euclidean optimization approaches), there is no publicly released Riemannian solvers for \ref{Prob_Ori}.  Considering the great effort needed in developing Riemannian optimization approaches and their corresponding solvers, it is challenging for these approaches to follow the progress and advances in nonconvex constrained optimization in the Euclidean space. 
	
	In brief, directly solving \ref{Prob_Ori} by Euclidean optimization approaches enjoys the great convenience from various available highly efficient algorithms and solvers. On the other hand, solving \ref{Prob_Ori} through Riemannian optimization approaches can achieve better efficiency by exploiting the structure of $\M$, while the involved geometrical materials may impose great difficulties in developing efficient Riemannian optimization solvers. Therefore, it is natural to consider how we can combine the advantages from both the Euclidean optimization approaches and Riemannian optimization approaches. More precisely, we are motivated to ask the following question:
	\begin{quote}
		\it 
		Can we directly apply Euclidean optimization solvers for \ref{Prob_Ori} while achieving better efficiency through exploiting the structure of the embedded submanifold $\M$ defined by constraints $c(x) = 0$? 
	\end{quote}

	Very recently, \cite{xiao2022dissolving} shows that the following smooth optimization problem 
	\begin{equation}
		\label{Eq_unconstrained_ori}
		\min_{x \in \Rn}\quad f(x), \qquad \text{s. t. }~  c(x) = 0,
	\end{equation}
	is equivalent to the unconstrained minimization of the following {\it constraint dissolving} function, 
	\begin{equation}
		\label{Eq_unconstrained_cdf}
		f(\A(x)) + \frac{\beta}{2} \norm{c(x)}^2. 
	\end{equation}
	Here $\beta \geq 0$ is the penalty parameter, and the constraint dissolving mapping $\A$ should satisfy the following assumptions.
	\begin{assumpt}{\bf Blanket assumptions on $\A$}
		\label{Assumption_2}	
		\begin{enumerate}
			\item $\A:\bb{R}^n \to \bb{R}^n$ is locally Lipschitz smooth in $\bb{R}^n$;
			\item $\A(x) = x$ holds for any $x \in \M$;
			\item The Jacobian of $c(\A(x))$ equals to $0$ for any $x \in \M$. That is, $\Ja(x) \Jc(x) = 0$  holds for any $x \in \M$, where $\Ja(x) \in \bb{R}^{n\times n}$ denotes the transposed Jacobian of $\A$ at $x$.
		\end{enumerate}
	\end{assumpt}
	The construction of the constraint dissolving mapping $\A$ for various important manifolds can be found in \cite[Section 4]{xiao2022dissolving}. 
	As proven in \cite{xiao2022dissolving}, \eqref{Eq_unconstrained_ori} and \eqref{Eq_unconstrained_cdf} have the same stationary points in a neighborhood of $\M$ with appropriately selected penalty parameter $\beta$. Moreover, \cite{xiao2022dissolving} shows that the construction of $\A$ can totally avoid the needs for any geometrical materials of $\M$. Therefore, various unconstrained optimization approaches can be directly applied to solve \eqref{Eq_unconstrained_ori} through \eqref{Eq_unconstrained_cdf}.

	Our motivation comes from the constraint dissolving approaches for Riemannian optimization \cite{xiao2022dissolving}. 
	The formulation of the constraint dissolving function in \eqref{Eq_unconstrained_cdf} motivates us to consider replacing $f$, $u$ and $v$ in \ref{Prob_Ori} by their corresponding constraint dissolving functions, and remove the constraints $c(x) = 0$ from \ref{Prob_Ori} through exact penalization. Thus we arrive at the following constraint dissolving problem (\ref{Prob_Pen}):
	\begin{equation}
		\tag{CDP}
		\label{Prob_Pen}
		\begin{aligned}
			\min_{x \in \bb{R}^n}\quad &h(x) := f(\A(x)) + \frac{\beta}{2} \norm{c(x)}^2\\
			\text{s.t.}\quad & \tilde{\ec}_i(x) := \ec_i(\A(x)) + \frac{\tau_i}{2} \norm{c(x)}^2 = 0, \quad i \in [N_E],\\
			&\tilde{\ic}_j(x) := \ic_j(\A(x)) + \frac{\gamma_j}{2} \norm{c(x)}^2 \leq 0, \quad j \in [N_I], 
		\end{aligned}
	\end{equation} 
	where $\beta$, $\{\tau_i\}$ and $\{\gamma_j\}$ are all non-negative penalty parameters. 
	As the construction of the constraint dissolving mapping $\A$ in \ref{Prob_Pen}
	is independent of the geometrical materials of $\M$ \cite{xiao2022dissolving},  we can directly apply various existing Euclidean solvers to  \ref{Prob_Pen},  without any computation of  geometrical materials such as retractions, vector transports, Riemannian differentials, etc.
	Furthermore, compared to \ref{Prob_Ori}, \ref{Prob_Pen} has eliminated the constraints $c(x) = 0$. Hence  solving the resulting problem \ref{Prob_Pen} by Euclidean solvers can potentially become more efficient.

	However, existing constraint dissolving approaches \cite{xiao2022dissolving,hu2022constraint,hu2022improved} are only developed for minimizing the  objective function $f$ over $\M$ without any additional constraints. How to establish the equivalence between \ref{Prob_Ori} and \ref{Prob_Pen}, especially in the presence of nonsmooth objective function and additional constraints, remains to be worked out. More importantly, as \ref{Prob_Ori} has different constraints as \ref{Prob_Pen}, the relationships between \ref{Prob_Ori} and \ref{Prob_Pen} on constraint qualifications,  stationary points, and KKT points, should be carefully analyzed.

	\paragraph{Contribution}
	In this paper, we propose the constraint dissolving approach for \ref{Prob_Ori} by  transforming it into \ref{Prob_Pen}, where the constraints $c(x) = 0$ is eliminated. We first prove the equivalence of the constraint qualifications between \ref{Prob_Ori} and \ref{Prob_Pen}, in the sense that a broad class of constraint qualifications for \ref{Prob_Ori} imply that all KKT points of \ref{Prob_Pen} are its first-order stationary points. 
	Moreover, we prove the equivalence between \ref{Prob_Ori} and \ref{Prob_Pen} in the sense that they have the same stationary points and Karush–Kuhn–Tucker (KKT) points over the feasible region under mild conditions. Additionally, we prove sharper results on the equivalence for the special case where $N_E = 0$ in \ref{Prob_Ori}. These results on the equivalence between \ref{Prob_Ori} and \ref{Prob_Pen} demonstrate that transforming \ref{Prob_Ori} into \ref{Prob_Pen} preserves the validity of constraint qualifications, while keeping the stationary points and KKT points of \ref{Prob_Ori} unchanged. Therefore, we can solve \ref{Prob_Ori} by directly employing various existing Euclidean approaches to solve \ref{Prob_Pen}.

	We perform numerical experiments on Riemannian center of mass problems and minimum balanced cut problems to demonstrate the strengths  of our proposed constraint dissolving approach. As a baseline, we compare our proposed constraint dissolving approach to the above-mentioned Euclidean optimization approaches by simply considering  $c(x) = 0$ as an extra set of equality constraints. Preliminary numerical experiments illustrate that applying existing Euclidean optimization approaches to \ref{Prob_Pen} can achieve better computational efficiency, especially in high dimensional cases. These numerical experiments further demonstrate that solving \ref{Prob_Ori} through \ref{Prob_Pen} can enjoy similar advantages offered by existing Riemannian optimization approaches, while avoiding the difficulties in determining the geometrical materials of the underlying manifolds and designing Riemannian optimization solvers.

	\paragraph{Organization}
	The outline of this paper is as follows. In Section 2, we fix the notations, definitions, and constants that are necessary for the proofs in this paper. We establish the theoretical properties of \ref{Prob_Pen} and discuss the special case with only inequality constraints. The proofs for the theoretical properties of \ref{Prob_Pen} are presented in the appendix. In Section 4, we present several  illustrative numerical examples to show that \ref{Prob_Pen} allows the straightforward implementations of various Euclidean optimization approaches.  We conclude the paper in the last section.

	\section{Notations, definitions and constants}
	\subsection{Notations}
	Let $\mathrm{range}(A)$ be the subspace spanned by the column vectors of matrix $A$, while $\norm{\cdot}_1$ and $\norm{\cdot}$ denote the $\ell_1$-norm and $\ell_2$-norm  of vectors or operators, respectively.
	The notations $\mathrm{diag}(A)$ and $\Diag(x)$
	stand for the vector formed by the diagonal entries of matrix $A$,
	and the diagonal matrix with the entries of $x\in\bb{R}^n$ as its diagonal, respectively. 
	We denote the smallest and largest eigenvalues of $A$ by $\lambda_{\mathrm{min}}(A)$ and $\lambda_{\max}(A)$, respectively. Besides, $\sigma_{l}(A)$ refers to the $l$-th largest singular value of $A$, and $\sigma_{\min}(A)$ refers to the smallest singular value of $A$. Furthermore, for any matrix $A \in \bb{R}^{n\times p}$, we use $A^\dagger$ to denote the pseudo-inverse of $A$. That it, $A^\dagger \in \bb{R}^{p\times n}$ satisfies $AA^\dagger A = A$, $A^\dagger AA^\dagger = A^\dagger$, and both $A^{\dagger} A$ and $A A^{\dagger}$ are symmetric \cite{golub1996Matrix}. 
	
	For the submanifold $\M$, we denote  
	$\Tx$ as the tangent space of $\M$ at $x$, which can be expressed as $\Tx := \{ d \in \bb{R}^n: \Jc(x)\tp d = 0  \}$.
	Moreover, $\Nx$ denotes the normal space of $\M$ at $x$, i.e., $\Nx := \{d \in \bb{R}^{n}:  d\tp u = 0, ~  \forall u \in \Tx \}$. 
	From the definitions of $\Tx$ and $\Nx$, it holds that $\Nx = \mathrm{range}(\Jc(x))$. 
	Additionally, for any $x \in \bb{R}^n$, we define the projection to $\M$ as $\mathrm{proj}(x, \M) := \mathop{\arg\min}_{y \in \M} ~ \norm{x-y}$. 
	It is worth mentioning that for any $x \in \bb{R}^n$, the optimality condition of the above problem leads to the fact that for any $y \in \mathrm{proj}(x, \M)$, it holds that $x - y \in \mathrm{range}(\Jc(y))$.
	Furthermore, $\mathrm{dist}(x, \M)$ refers to the distance between $x$ and $\M$, i.e. $
		\mathrm{dist}(x, \M) := \mathop{\inf}_{y \in \M} ~ \norm{x-y}$.

	The transposed Jacobian of mappings $\A$ and $c$ are denoted as $\Ja(x) \in \bb{R}^{n\times n}$ and $\Jc(x) \in \bb{R}^{n\times p}$, respectively. That is, let $c_i$ and $\A_{i}$ denotes the $i$-th coordinate of the mapping $c$ and $\A$ respectively, then $\Jc$ and $\Ja$ are defined as 
	\begin{equation*}
		\Jc(x) := \left[ \nabla c_1(x),..., \nabla c_p(x) \right], \quad \text{and} \quad \Ja(x):= \left[\nabla \A_1(x), ..., \nabla \A_n(x) \right]. 
	\end{equation*}
	Similarly, we define $\Ju$ and $\Jv$ as the transported Jacobian for $u$ and $v$, respectively. 
	Moreover, for a given subset $\ca{G}$ of $\bb{R}^n$, we define $ \Ja(x)\ca{G} := \{ \Ja(x)d: d \in \ca{G} \}$. We set $$\A^{k}(x) := \underbrace{\A(\A(\cdots\A(x)\cdots))}_{k \text{ times}},$$
	for $k \geq 1$, and define $\A^0(x) := x$, $\A^{\infty}(x):= \lim\limits_{k \to +\infty} \A^k(x)$. Furthermore, we denote $g(x) := f(\A(x))$ in the rest of this paper. Additionally, we denote the closed ball at $x$ with radius $\rho$ as $\ca{B}_{x, \rho} := \{ y\in \Rn: \norm{y-x} \leq \rho \}$. Finally, for any positive integer $n$, we use the notation $[n]:= \{1,2,...,n\}$. We denote $\KA$ as the feasible region of \ref{Prob_Pen}, i.e. $\KA :=  \{ x \in \bb{R}^n: \tilde{\ec}(x) = 0, \tilde{\ic}(x) \leq 0 \}$, and denote $\ca{F}(x)$ as the active index set of the inequality constraints in \ref{Prob_Ori}, i.e. $\ca{F}(x) := \{ j \in [N_I]: \ic_j(x) = 0 \}$. Similarly, we denote the active index set of the inequality constraints in \ref{Prob_Pen} as $\ca{F}_A(x) := \{ j \in [N_I]: \tilde{\ic}_j(x) = 0\}$.   Note that for any $x \in \M$, we have $\ca{F}(x) = \ca{F}_A(x)$, since $\tilde{\ic}_j(x) = \ic_j(x)$ holds for any $x \in \M$ and $j \in [N_I]$. Furthermore, it is easy to verify that $\K\subseteq \KA$.

	\subsection{Preliminaries}
	\subsubsection{Subdifferential and regularity}
	\begin{defin}
		
		The generalized directional derivative of $f$ at $x \in \bb{R}^{n}$  in the direction $d$, denoted by $f^\star(x, d)$, is defined as 
		\begin{equation}
			f^\star(x, d) = \mathop{\lim\sup}_{y\to x, t \downarrow  0}~ \frac{f(y +td) - f(y)}{t}. 
		\end{equation}
		The generalized gradient or the (Clarke) subdifferential of $f$ at $x \in \bb{R}^{n\times p}$, denoted by $\partial f(x)$, is defined as 
		\begin{equation}
			\partial f(x) := \{ w \in \bb{R}^{n} : \inner{w, d} \leq f^\star (x, d), \text{ for all } d \in \bb{R}^{n} \}.
		\end{equation}
	\end{defin}
	

	\begin{defin}
		\label{Defin_Clarke_regular}
		We say that $f$ is (Clarke) regular at $x \in \bb{R}^{n}$ if for every direction $d$, the one-sided directional derivative 
		\begin{equation}
			f'(x,d) = \lim_{t \downarrow 0} \frac{f(x + td) - f(x)}{t} 
		\end{equation}
		exists and $f'(x, d) = f^\star(x, d)$.
	\end{defin}

%

	\subsubsection{Constraint qualifications and  optimality conditions}
	In this subsection, we present the definitions on the constraint qualifications for the constrained optimization problems \eqref{Prob_Ori} and \eqref{Prob_Pen}, respectively. 
	It is worth mentioning that both \ref{Prob_Ori} and \ref{Prob_Pen} are constrained optimization problems in $\Rn$, hence their constraint qualifications can be defined in the same manner.  Recall that $\K$ and $\KA$ denote the feasible region of \ref{Prob_Ori} and \ref{Prob_Pen}  respectively, we have the following definitions on the constraint qualifications for \eqref{Prob_Ori} and \eqref{Prob_Pen}. 

	\begin{defin}
		For any closed cone $\ca{C}$ in $\Rn$, its polar cone $\ca{C}^\circ$ is defined as
		\begin{equation*}
			\ca{C}^\circ := \{d \in \Rn: \inner{d, w} \leq 0 , ~\forall w \in \ca{C} \}. 
		\end{equation*}
	\end{defin}

	\begin{defin}
		\label{Defin_tangent_cones}
		For any $x \in \K$, the (Bouligand) tangent cone with respect to $\K$ is defined as
		\begin{equation*}
			\TK(x) := \left\{ d \in \Rn: \exists \xk \in \K \to x, ~t_k \downarrow 0, \text{ s.t.  } d = \lim_{k \to +\infty} \frac{x_k - x}{t_k} \right\}.
		\end{equation*}
		Similarly, for any $x \in \KA$,  the (Bouligand) tangent cone with respect to $\KA$ is defined as
		\begin{equation*}
				\TKA(x) := \left\{ d \in \Rn: \exists \xk \in \KA \to x, ~t_k \downarrow 0, \text{ s.t.  } d = \lim_{k \to +\infty} \frac{x_k - x}{t_k} \right\}.
		\end{equation*}
	\end{defin}
	
	\begin{defin}
		\label{Defin_linearize_cones}
		For any $x \in \K$, the linearizing cone with respect to $\K$ is defined as
		\begin{equation*}
			\TKlin(x) := \left\{ d \in \Rn: \inner{d, \nabla c_l(x)} = 0,  \inner{d, \nabla u_i(x)} = 0, \inner{d, \nabla v_j(x)} \leq 0, l\in [p], i \in [N_E], j\in \ca{F}(x) \right\}.
		\end{equation*}
		Similarly, for any $x \in \KA$, the linearizing cone with respect to $\KA$ is defined as
		\begin{equation*}
			\TKAlin(x) := \left\{ d \in \bb{R}^n: \inner{d,  \nabla \tilde{\ec}_i(x)} = 0, \inner{d,  \nabla  \tilde{\ic}_j(x)} \leq 0, i \in [N_E], j\in \ca{F}_A(x) \right\}
		\end{equation*}
	\end{defin}
From Definition \ref{Defin_linearize_cones}, it can be shown that for any $x\in\K$,  $\TK(x) \subset \TKlin(x)$ and
$$
(\TKlin(x))^\circ = \left\{
\sum_{i\in[p]} \rho_l\nabla {c}_l(x) 
+ \sum_{i\in[N_E]} \lambda_i \nabla {u}_i(x) 
+\sum_{j\in[N_I]} \mu_j \nabla {v}_j(x)\;\Big| 
\begin{array}{l}
\rho\in\mathbb{R}^p,\;
\lambda \in \mathbb{R}^{N_E}, 
\\
\mu\in\mathbb{R}^{N_I}_+, \mu_j = 0\; \forall\; j \not\in \ca{F}(x)
\end{array} \right\}.
$$
Similarly, it can also be shown that for any $x\in\KA$,
$\TKA(x) \subset \TKAlin(x)$ and
$$
(\TKAlin(x))^\circ = \Big\{ \sum_{i\in[N_E]} \lambda_i \nabla \tilde{u}_i(x) 
+\sum_{j\in[N_I]} \mu_j \nabla \tilde{v}_j(x)\mid \lambda \in \mathbb{R}^{N_E}, \mu\in\mathbb{R}^{N_I}_+, \mu_j = 0\; \forall\; j \not\in \ca{F}_A(x) \Big\}.
$$

	\begin{defin}
		\label{Defin_CQ}
		For any given $x \in \K$, we define the following constraint qualifications at $x$ for \ref{Prob_Ori}.
		\begin{itemize}
			\item The linear independence constraint qualification (LICQ) with respect to \ref{Prob_Ori} holds at $x$ if $\{\nabla c_l(x): l\in [p]\}\cup\{\nabla \ec_i(x):  i \in [N_E] \} \cup \{\nabla \ic_j(x): j \in \ca{F}(x)\}$ is a linearly independent set in $\bb{R}^n$. 
			\item The Mangasarian-Fromovitz constraint qualification (MFCQ) with respect to \ref{Prob_Ori} holds at $x$ if $\{\nabla c_l(x): l\in [p]\}\cup\{\nabla \ec_i(x): 
			i \in [N_E]\}$ is linearly independent and there exists $d \in \Rn$ such that 
			\begin{align*}
				&\inner{d, \nabla c_l(x)} = 0, ~ \forall l  \in [p],\\
				&\inner{d, \nabla \ec_i(x)} = 0, ~ \forall i  \in [N_E],\\
				&\inner{d, \nabla \ic_j(x)} \leq 0, ~ \forall j \in \ca{F}(x).
			\end{align*}
			\item The GCQ with respect to \ref{Prob_Ori} holds at $x$ if $\TK(x)^\circ = \TKlin(x)^\circ$.
		\end{itemize}
	
		Similarly, for any given $x \in \KA$, we define the following constraint qualifications at $x$ for \ref{Prob_Pen}.
		\begin{itemize}
			\item The linear independence constraint qualification (LICQ) with respect to \ref{Prob_Pen} holds at $x$ if $\{\nabla \tilde{\ec}_i(x): i \in [N_E] \} \cup \{\nabla \tilde{\ic}_j(x): j \in \ca{F}(x)\}$ is a linearly independent set in $\bb{R}^n$. 
			\item The Mangasarian-Fromovitz constraint qualification (MFCQ) with respect to \ref{Prob_Pen} holds at $x$ if $\{\nabla \tilde{\ec}_i(x): i \in [N_E]\}$ is linearly independent and there exists $d \in \Rn$ such that 
			\begin{align*}
				&\inner{d, \nabla \tilde{\ec}_i(x)} = 0, ~ \forall i  \in [N_E],\\
				&\inner{d, \nabla \tilde{\ic}_j(x)} \leq 0, ~ \forall j \in \ca{F}_A(x).
			\end{align*}
			\item The GCQ with respect to \ref{Prob_Pen} holds at $x$ if $\TKA(x)^\circ = \TKAlin(x)^\circ$.
		\end{itemize}
	\end{defin}

	It is worth mentioning that the constraint qualifications in Definition \ref{Defin_CQ} are equivalent to the Riemannian constraint qualifications defined in \cite[Definition 3.12]{bergmann2019intrinsic}, as discussed in \cite[Remark 3.13]{bergmann2019intrinsic}. 
	Furthermore, based on the definition of tangent cones in Definition \ref{Defin_tangent_cones},  we define the stationary points for \ref{Prob_Ori} and \ref{Prob_Pen} as follows.
	\begin{defin}
		\label{Defin_FOSP}
		We say $x \in \K$ is a first-order stationary point of \ref{Prob_Ori} if  it satisfies 
		\begin{equation}
			0 \in \partial f(x) + \TK(x)^\circ. 
		\end{equation} 
		Moreover, $x \in \KA$ is a first-order stationary point of \ref{Prob_Pen} if 
		\begin{equation}
			0 \in \partial h(x) + \TKA(x)^{\circ}. 
		\end{equation}
	\end{defin}

	\subsubsection{KKT conditions}
	In this subsection, we introduce the definition of KKT conditions and optimality conditions for both \ref{Prob_Ori} and \ref{Prob_Pen}.
	
	\begin{defin}
		For any given $x \in \Rn$, the Lagrangian of \ref{Prob_Ori} is defined as 
		\begin{equation}
			\LNLP(x, \lambda, \rho, \mu):= f(x) + \rho\tp c(x) + \lambda\tp u(x)   + \mu \tp v(x),
		\end{equation}
		whose the partial subdifferential with respect to variable $x$ is defined as
		\begin{equation}
			\partial_x \LNLP(x, \rho, \lambda, \mu):= \partial f(x) + \sum_{l \in [p]} \rho_l  \nabla c_l(x) + \sum_{i \in [N_E]} \lambda_i  \nabla u_i(x)  + \sum_{ j\in [N_I] } \mu_j \nabla v_j(x).
		\end{equation}
		Moreover, for any given $x \in \bb{R}^n$, the Lagrangian of \ref{Prob_Pen} is defined as 
		\begin{equation}
			 \LCDP (x, \lambda, \mu) := h(x) + \lambda\tp \tilde{\ec}(x) + \mu\tp \tilde{\ic}(x),
		\end{equation}
		whose partial subdifferential with respect to variable $x$ is defined as
		\begin{equation}
			\partial_x \LCDP(x,\lambda, \mu):= \partial h(x) + \sum_{ i\in [N_E]} \lambda_i \nabla \tilde{\ec}_i(x) + \sum_{j \in [N_I]} \mu_j \nabla \tilde{\ic}_j(x) .
		\end{equation} 
	\end{defin}
	
	Based on the  Lagrangians of \ref{Prob_Ori} and \ref{Prob_Pen}, we present the definition of KKT points as follows. 
	\begin{defin}
		\label{Defin_KKT_NLP}
		For any given $x \in \M$, we say $x$ is a KKT point to \ref{Prob_Ori} (with  multipliers $\rho$, $\lambda$ and $\mu$) if 
		\begin{equation}
			\begin{aligned}
				&0 \in \partial_x  \LNLP (x, \rho,\lambda,  \mu),\\
				& c(x) = 0,  ~u(x) = 0, ~ v(x) \leq 0, ~\mu \geq 0,~ \mu\tp v(x) = 0. 
			\end{aligned}
		\end{equation}
	\end{defin}

	\begin{defin}
		\label{Defin_KKT_CDP}
		For any given $x \in \bb{R}^n$, we say $x$ is a KKT point to \ref{Prob_Pen} (with  multipliers $\lambda$ and $\mu$) if 
		\begin{equation}
			\begin{aligned}
			&0 \in \partial_x \LCDP(x, \lambda, \mu),\\
			& \tilde{\ec}(x) = 0,  ~\tilde{\ic}(x) \leq 0, ~\mu \geq 0,~ \mu\tp \tilde{\ic}(x) = 0. 
		\end{aligned}
		\end{equation}
	\end{defin}

	%
	%

	\subsection{Constants}
	In this paper, for any given $x \in \M$, we choose a positive scalar $\rho_x$ such that 
	\begin{equation*}
		\rho_x := \mathop{\arg\max}_{0 < \rho \leq 1}~ \left\{\rho:  \inf \left\{\sigma_{\min}(\Jc(y)): y \in \ca{B}_{x, \rho}\right\} \geq \frac{1}{2} \sigma_{\min}(\Jc(x))\right\}. 
	\end{equation*}
	Then for any given $x \in \M$,  based on the definition of $\rho_x$, we define $\Theta_x:= \{ y \in \ca{X}: \norm{y-x} \leq \rho_x \}$ and set several constants as follows:
	\begin{itemize}
		\item $\Lsc  := \sigma_{\min}(\Jc(x))$; 
		\item $\Lf  := \sup_{y \in \Theta_x}\sup_{w \in \partial f(\A(y))}~ \norm{w}$;
		\item $\Mu := \sup_{y \in \Theta_x } \norm{\Ju(y)} $;
		\item $\Mv := \sup_{y \in \Theta_x} \norm{\Jv(y)} $;
		\item $\Mc  := \sup_{y \in \Theta_x}   \norm{\Jc(y)} $;
		\item $\Ma  := \sup_{y \in \Theta_x}  \norm{\Ja(y)}$;
		\item $\Lc := \sup_{y, z \in \Theta_x, y\neq z} \frac{\norm{\Jc(y) - \Jc(z) }}{\norm{y-z}} $;
		\item $\La := \sup_{y, z \in \Theta_x, y\neq z} \frac{\norm{\Ja(y) - \Ja(z) }}{\norm{y-z}} $;
		\item $\Lac := \sup_{y, z \in \Theta_x, y\neq z} \frac{\norm{ \Ja(y)\Jc(\A(y)) - \Ja(z)\Jc(\A(z)) }}{\norm{y-z}} $.
	\end{itemize}
	It is worth mentioning that all the aforementioned constants are independent of the penalty parameter $\beta$. 
	Moreover, for any given $x \in \bb{R}^n$ with multipliers $(\lambda, \mu)$, we set $M_{x,\lambda,\mu} :=  \Lf+ \norm{\lambda}_1\Mu +\norm{\mu}_1\Mv $. 
	Furthermore, based on these constants, we set 
	\begin{itemize}
		\item $\varepsilon_x := \min \left\{ \frac{\rho_{x}}{2}, \frac{\Lsc}{32\Lc(\Ma+1)}, \frac{\Lsc ^2}{8\Lac\Mc } \right\}$;
		\item $\Omegax{x} := \left\{ y \in \Rn: \norm{y - x}  \leq \varepsilon_x \right\}$;
		\item $\BOmegax{x} := \left\{ y \in \Rn: \norm{y - x}  \leq \frac{\Lsc \varepsilon_x}{4\Mc(\Ma +1) + \Lsc}  \right\}$; 
		\item $\Omega:= \bigcup_{x \in \M} \Omegax{x}$;
		\item $\overline{\Omega}:= \bigcup_{x \in \M} \BOmegax{x}$.
	\end{itemize}
	Note that $\varepsilon_x \leq \frac{\Lsc ^2}{8\Lac\Mc }$ automatically implies that $\varepsilon_x \leq \frac{\Lsc}{4\Lac}$ since $\Lsc \leq \Mc$. Additionally, it is worth mentioning that for any given $x \in \M$, Assumption \ref{Assumption_1} guarantees that $\Lsc > 0$ and $\varepsilon_x > 0$. Furthermore, from the definition of $\overline{\Omega}$, we can conclude that $\BOmegax{x} \subset \Omegax{x} \subset \Theta_x$ holds for any given $x \in \M$, and $\M$ lies in the interior of $\overline{\Omega}$.

	\section{Theoretical Properties}
	In this section, we present the theoretical properties of \ref{Prob_Pen}. We first present the explicit expression of $\partial h(x)$ in the following proposition. We omit its proof since it can be established by direct calculation. 
	\begin{prop}
		\label{Prop_partial_h}
		For any $x \in \bb{R}^n$, it holds that 
		\begin{align}
			\partial h(x) ={}& \Ja(x) \partial f(\A(x)) + \beta \Jc(x) c(x),\\
			\nabla \tilde{\ec}_i(x) ={}& \Ja(x) \nabla \ec_i(\A(x)) + \tau_i \Jc(x) c(x),\\
			\nabla \tilde{\ic}_j(x) ={}& \Ja(x) \nabla \ic_j(\A(x)) + \gamma_j \Jc(x) c(x).
		\end{align}
		Moreover, for any $x \in \M$, $i \in [N_E]$ and $j \in [N_I]$,  it holds that 
		\begin{equation*}
		\partial h(x) = \Ja(x) \partial f(x), ~ \nabla \tilde{\ec}_i(x) = \Ja(x) \nabla \ec_i(x), ~ \text{and}~\nabla \tilde{\ic}_j(x) = \Ja(x) \nabla \ic_j(x).
		\end{equation*}
	\end{prop}

	\subsection{Constraint qualifications}

	In this subsection, we analyze the relationship on the constraint qualifications between \ref{Prob_Ori} and \ref{Prob_Pen}. We show that LICQ and MFCQ with respect to \ref{Prob_Ori} imply the validity of the corresponding constraint qualifications with respect to \ref{Prob_Pen}, while GCQ with respect to \ref{Prob_Ori} implies that the KKT points of \ref{Prob_Pen} are all first-order stationary points of \ref{Prob_Pen}. Fugure \ref{Figure_CQs} summarizes the main results in this subsection. 
	
	\begin{figure}[htbp]
		\begin{equation*}
			\begin{tikzcd}[row sep=huge,column sep=normal]
				\text{LICQ for \ref{Prob_Ori}}  \arrow[d, Rightarrow, "\text{Prop. \ref{Prop_Equivalence_LICQ}}"] \arrow[r, Rightarrow] & \text{MFCQ for \ref{Prob_Ori}} \arrow[d, Rightarrow, "\text{Prop. \ref{Prop_Equivalence_MFCQ}}"] \arrow[r, Rightarrow] & \text{GCQ for \ref{Prob_Ori}} \arrow[d, Rightarrow, "\text{Prop. \ref{Prop_CQ_equivalence_GCQ}}"]\\
				\text{LICQ for \ref{Prob_Pen}}  \arrow[r, Rightarrow] & \text{MFCQ for \ref{Prob_Pen}} \arrow[r, Rightarrow] & \text{\eqref{Eq_Prop_CQ_equivalence_GCQ} for \ref{Prob_Pen}}\\
			\end{tikzcd}
		\end{equation*}
		\vspace{-15mm}
		\caption{Relationships on  the constraint qualifications between \ref{Prob_Ori} and \ref{Prob_Pen}. The fact that \eqref{Eq_Prop_CQ_equivalence_GCQ} can be regarded as a constraint qualification for \ref{Prob_Pen} is shown in Theorem \ref{The_CQ_fin}. Here ``Prop.'' is the abbreviation for ``Proposition''.}
		\label{Figure_CQs}
	\end{figure}
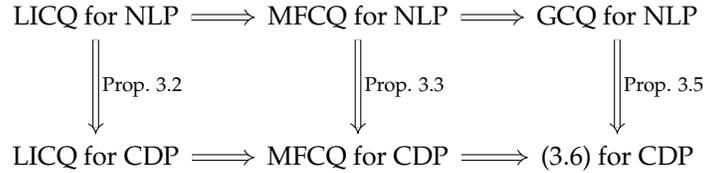
	
	The proofs for Proposition \ref{Prop_Equivalence_LICQ} -- Proposition \ref{Prop_CQ_equivalence_GCQ} are presented in Section \ref{Subsection_proofs_CQs}. 
	
	\begin{prop}
		\label{Prop_Equivalence_LICQ}
		For any $x \in \K$, suppose the LICQ with respect to \ref{Prob_Ori} holds at $x$, then  the  LICQ with respect to \ref{Prob_Pen} holds at $x$.
	\end{prop}

	\begin{prop}
		\label{Prop_Equivalence_MFCQ}
		For any $x \in \K$, suppose the MFCQ with respect to \ref{Prob_Ori} holds at $x$, then  the  MFCQ with respect to \ref{Prob_Pen} holds at $x$.
	\end{prop}

	\begin{lem}
		\label{Le_relationship_ACQ}
		For any $x \in \K$, it holds that 
		\begin{align}
			&\TKlin(x) = \Ja(x)\tp \TKAlin(x), \quad \TKAlin(x)^\circ \cap \mathrm{range}(\Ja(x)) = \Ja(x)\TKlin(x)^\circ, \label{Eq_Le_relationship_ACQ_0} \\
			&\TK(x) \subseteq \Ja(x)\tp \TKA(x), \quad \TKA(x)^\circ \cap \mathrm{range}(\Ja(x)) \subseteq \Ja(x)\TK(x)^\circ.\label{Eq_Le_relationship_ACQ_1} 
		\end{align}
	\end{lem}

	In the rest of this subsection, we aim to illustrate that the GCQ for \ref{Prob_Ori} implies that all the KKT points of \ref{Prob_Pen} are first-order stationary points of \ref{Prob_Pen}. We first show the relationships between $\TKA(x)$ and $\TKAlin(x)$ when GCQ for \ref{Prob_Ori} holds at $x \in \M$. 
%
%
%
	
	\begin{prop}
		\label{Prop_CQ_equivalence_GCQ}
		For any $x \in \K$, suppose GCQ for \ref{Prob_Ori} holds at $x$, i.e., $\TK(x)^\circ = \TKlin(x)^\circ$. Then it holds that 
		\begin{equation}
			\label{Eq_Prop_CQ_equivalence_GCQ}
			\TKA(x)^\circ \cap \mathrm{range}(\Ja(x)) =  \TKAlin(x)^\circ \cap \mathrm{range}(\Ja(x)).
		\end{equation}
	\end{prop}

	As mentioned in various existing works \cite{Nocedal1999Numerical}, GCQ is usually recognized as one of the weakest constraint qualifications in the sense that a great number of constraint qualifications for \ref{Prob_Ori} leads to the validity of GCQ for \ref{Prob_Ori}. Interested readers could refer to \cite{wang2013constraint} for more details on the relationships among various constraint qualifications.  In the following theorem, we illustrate that Proposition \ref{Prop_CQ_equivalence_GCQ} ensures that KKT conditions holds at any first-order stationary point of \ref{Prob_Pen}. We present the proof for Theorem \ref{The_CQ_fin} in Section \ref{Subsection_proofs_CQs}.  
	
	\begin{theo}
		\label{The_CQ_fin}
		For any $x \in \K$, suppose $\TKA(x)^\circ \cap \mathrm{range}(\Ja(x)) =  \TKAlin(x)^\circ \cap \mathrm{range}(\Ja(x))$, then $x$ is a first-order stationary point of \ref{Prob_Pen} if and only if $x$ is a KKT point of \ref{Prob_Pen}. 
	\end{theo}

	\subsection{Stationarity}
	In this subsection, we aim to analyze the relationship between  \ref{Prob_Ori} and   \ref{Prob_Pen} in the aspect of their first-order stationary points and  KKT points. Figure \ref{Figure_Stat} summarizes the main results of this subsection. In addition, the proofs for Theorem \ref{The_equivalence_feasible}, Theorem \ref{The_equivalence_FOSP} and Theorem \ref{The_fval_reduction} are presented in Section \ref{Subsection_proofs_equivalence}. 
	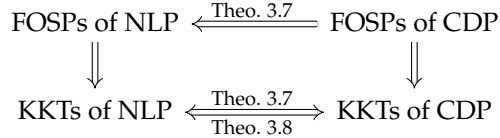
\begin{figure}[htbp]
		\begin{equation*}
			\begin{tikzcd}[row sep=normal,column sep=huge]
				\text{FOSPs of \ref{Prob_Ori}} \arrow[d, Rightarrow]  \arrow[r, Leftarrow, "\text{Theo. \ref{The_equivalence_feasible} }"] &\text{FOSPs of \ref{Prob_Pen}}  \arrow[d, Rightarrow]\\
				\text{KKTs of \ref{Prob_Ori}}  \arrow[r, Leftrightarrow, "\text{Theo. \ref{The_equivalence_feasible} }", "\text{Theo. \ref{The_equivalence_FOSP} }" swap] &\text{KKTs of \ref{Prob_Pen}}\\
			\end{tikzcd}
		\end{equation*}
	\vspace{-10mm}
	\caption{Relationships on  the first-order stationary points (FOSPs) and KKT points (KKTs) between \ref{Prob_Ori} and \ref{Prob_Pen}. Here ``Theo.'' is the abbreviation for ``Theorem''.  }
	\label{Figure_Stat}
	\end{figure}

	\begin{theo}
		\label{The_equivalence_feasible}
		For any given $x \in \K$, it holds that
		\begin{enumerate}
			\item if $x$  is a first-order stationary point of \ref{Prob_Pen}, then $x$ is a first-order stationary point of \ref{Prob_Ori};
			\item $x$ is a KKT point to \ref{Prob_Ori} if and only if $x$ is a KKT point to \ref{Prob_Pen}. 
		\end{enumerate}
	\end{theo}

	Applying  Euclidean optimization approaches to solve \ref{Prob_Pen} usually yields an infeasible sequence that converges to a KKT point of \ref{Prob_Ori}. Therefore, it is of great importance to analyze the relationship between the KKT points of \ref{Prob_Ori} and the  KKT points of \ref{Prob_Pen}. The following theorem illustrates that under mild conditions with sufficiently large $\beta$, any KKT point $y$ to \ref{Prob_Pen} in $\BOmegax{x}$ for some $x \in \M$ is a  KKT point of \ref{Prob_Ori}. Note that the difference between the result here and that of Theorem \ref{The_equivalence_feasible} is that we do not assume the KKT point $y \in \K$ in Theorem \ref{The_equivalence_FOSP}.

	\begin{theo}
		\label{The_equivalence_FOSP}
		For any given $x \in \M$, suppose $y \in \BOmegax{x}$ is a KKT point of \ref{Prob_Pen} with multipliers $\lambda$ and $\mu$, then it holds that 
		\begin{equation} \label{thm3.8-eq-1}
			\begin{aligned}
				&\mathrm{dist}(0, \partial_x \LCDP (y,\lambda,\mu) )\\
				&\geq \left( \frac{ \Lsc }{4(\Ma + 1)}\Big(\beta + \sum_{i\in [N_E]} \lambda_i \tau_i + \sum_{j \in [N_I]} \mu_j \gamma_j \Big) - \frac{4\La\Mxlambdamu}{\Lsc } \right)\norm{c(\y)}
			\end{aligned}
		\end{equation}
		Furthermore, suppose  
		\begin{equation}
			\beta + \sum_{i \in [N_E]} \lambda_i \tau_i + \sum_{j \in [N_I]} \mu_j \gamma_j\geq \frac{32\La(\Ma  + 1)\Mxlambdamu }{\Lsc^2 }
		\end{equation}
		holds at $y$, then $y$ is a KKT point of \ref{Prob_Ori}. 
	\end{theo}

	\begin{rmk}
		Theorem \ref{The_equivalence_FOSP} provides a theoretical lower-bound for $\beta$ that depends on the multipliers $\lambda$ and $\mu$. In various existing Euclidean approaches, the multipliers $\{\lambda_i\}_{i\in[N_E]}$ are assumed or proved to be uniformly bounded. In these cases, we can immediately achieve a theoretical lower-bound for $\beta$. 
		
		On the other hand, when the uniform boundness of the multipliers $\{\lambda_i\}_{i\in[N_E]}$ is not guaranteed,   some existing works can estimate an upper-bound for the multipliers \cite{mangasarian1985computable}. Therefore, we can adaptively adjust $\beta$ to force it to satisfy the theoretical lower-bound in Theorem \ref{The_equivalence_FOSP}, and hence ensure that the applied optimization approach can find a KKT point of \ref{Prob_Ori}. 
		
		Furthermore, for the special cases where $N_E = 0$ (i.e., no additional equality constraints in \ref{Prob_Ori}), we can prove that the theoretical lower bound for $\beta$ in \ref{Prob_Pen} is independent of the multipliers. We put the detailed discussion in Section 3.3. 
	\end{rmk}
	
	As illustrated in \cite{xiao2022dissolving}, we know that the constraint dissolving operator $\A$ can quadratically reduce the feasibility violation of any infeasible point $y \in \BOmegax{x}$, hence $\A^{\infty}(y)$ is feasible. 
	In the following theorem, we prove the relationships  between $y$ and $\A^{\infty}(y)$ in terms of the function values and derivatives, which is of great importance in characterizing the properties of  \ref{Prob_Pen} at those infeasible points in $\BOmegax{x}$.
	\begin{theo}
		\label{The_fval_reduction}
		For any given $x \in \KA$, suppose 
		$y \in \BOmegax{x}$ is a KKT point of \ref{Prob_Pen} with multipliers $\lambda$ and $\mu$, and 
		\begin{equation}
			\beta + \sum_{i \in [N_E]} \lambda_i \tau_i + \sum_{j \in [N_I]} \mu_j \gamma_j \geq \frac{8\Mxlambdamu (\Ma +1)\Lac }{\Lsc ^3}.
		\end{equation}
		Then it holds that 
		\begin{equation}
			\begin{aligned}
				\LCDP(y, \lambda, \mu) \geq{}& \LCDP( \A(y), \lambda, \mu ),\\
				\LCDP(y, \lambda, \mu) \geq{}& \LCDP( \A^{\infty}(y), \lambda, \mu ).
			\end{aligned}
		\end{equation}
	\end{theo}

	\subsection{Inequality Constrained Case}
	In this subsection, we consider the special case of \ref{Prob_Ori} where the additional constraints $u(x) = 0$ are absent. More precisely, we analyze the equivalence between \ref{Prob_Ori} and \ref{Prob_Pen} under the following assumption. 
	\begin{assumpt}
		\label{Assumption_Ineq}
		$N_E = 0$ in \ref{Prob_Ori}. 
	\end{assumpt} 
	
	As the additional equality constraints $u(x) = 0$ are absent, the Lagrangian of \ref{Prob_Pen} is denoted as $\LCDP(x, \mu)$ for simplicity in this subsection. Moreover, we make the following condition on the threshold value for $\beta$ and $\{\gamma_j\}$ for a given $x \in \M$, which is independent to the corresponding multiplier $\mu$.
	\begin{cond}
		\label{Cond_beta_gamma}
		Given any $x \in \M$, the parameters $\beta$ and $\{\gamma_j\}_{j \in [N_I]}$ in \ref{Prob_Pen} satisfy the following inequalities,
		\begin{equation}
			\beta \geq \frac{64\Lf(\Ma +1)(\Lac + \Lsc \La) }{\Lsc ^3}, \quad  \inf_{j \in [N_I]}~\gamma_j \geq \frac{32\La \Mv(\Ma + 1)}{\Lsc^2}. 
		\end{equation}
	\end{cond}

	\begin{rmk}
		When $\M$ is compact, there exists a finite set $\ca{I}\subset \M$ such that $\M \subseteq \bigcup_{x \in \ca{I}} \Theta_x$. Therefore, we can choose uniform positive lower bounds for $\Lsc$ and $\varepsilon_x$, while finding uniform upper bounds for all the other aforementioned constants. Specifically, when $\M$ is compact, we can choose a uniform upper bound for the penalty parameters $\beta$ and $\{\gamma_j\}$ for \ref{Prob_Pen} under Assumption \ref{Assumption_Ineq}. 
	\end{rmk}

	We first study the equivalence on constraint qualifications. The following lemma illustrates that for any $x \in \K$ with $\beta$ and $\{\gamma_j\}$ satisfying Condition \ref{Cond_beta_gamma}, we can obtain sharper results on the relationship between $\TK(x)$ and $\TKA(x)$. The main results of this subsection is summarized in Figure \ref{Figure_Stat_Ineq}.  The proof for Lemma \ref{Le_relationship_ACQ_ineq}, Theorem \ref{The_equivalence_feasible_Ineq}, Theorem \ref{The_equivalence_Ineq_FOSP}, and Corollary \ref{Coro_equivalence_Ineq_KKT} can be found in Section \ref{Subsection_proofs_ineqs}. 
	
	\begin{figure}[htbp]
		\begin{equation*}
			\begin{tikzcd}[row sep=normal,column sep=huge]
				\text{FOSPs of \ref{Prob_Ori}} \arrow[d, Rightarrow]  \arrow[r, Leftrightarrow, "\text{Theo. \ref{The_equivalence_feasible_Ineq} }"] &\text{FOSPs of \ref{Prob_Pen}}  \arrow[d, Rightarrow]\\
				\text{KKTs of \ref{Prob_Ori}}  \arrow[r, Leftrightarrow, "\text{Theo. \ref{The_equivalence_feasible} }", "\text{Coro. \ref{Coro_equivalence_Ineq_KKT} }" swap] &\text{KKTs of \ref{Prob_Pen}}\\
			\end{tikzcd}
		\end{equation*}
		\vspace{-10mm}
		\caption{Relationships on  the first-order stationary points (FOSPs) and KKT points (KKTs) between \ref{Prob_Ori} and \ref{Prob_Pen} under Assumption \ref{Assumption_Ineq}. Here ``Theo.'' is the abbreviation for ``Theorem'', while ``Coro.'' is the abbreviation for ``'Corollary''.   }
		\label{Figure_Stat_Ineq}
	\end{figure}
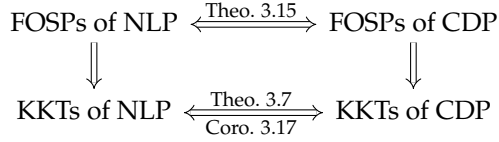

	\begin{lem}
		\label{Le_relationship_ACQ_ineq}
		Suppose Assumption \ref{Assumption_Ineq} holds, and the penalty parameters $\beta$ and $\{\gamma_j\}$ in \ref{Prob_Pen} satisfy Condition \ref{Cond_beta_gamma},  then it holds for any $x \in \K$ that 
		\begin{align}
			&\TK(x) = \Ja(x)\tp \TKA(x),\\
			&\TKA(x)^\circ \cap \mathrm{range}(\Ja(x)) = \Ja(x)\TK(x)^\circ.
		\end{align}
	\end{lem}

	In the following theorem, we use  Lemma \ref{Le_relationship_ACQ_ineq} to obtain a sharper result on the equivalence of first-order stationary points between \ref{Prob_Ori} and \ref{Prob_Pen}, in the sense that the lower bound for  $\beta$ and $\{\gamma_j\}$ are independent of the multiplier $\mu$. 
	\begin{theo}
		\label{The_equivalence_feasible_Ineq}
		For any given $x \in \K$, suppose Assumption \ref{Assumption_Ineq} holds, and the penalty parameters $\beta$ and $\{\gamma_j\}$ in \ref{Prob_Pen} satisfy Condition \ref{Cond_beta_gamma}. Then $x$ is a first-order stationary point of \ref{Prob_Ori} if and only if $x$ is a first-order stationary point of \ref{Prob_Pen}. 
	\end{theo}
	
	Based on Theorem \ref{The_equivalence_feasible_Ineq}, we prove that the equivalence of KKT points in Theorem \ref{The_equivalence_FOSP} can be further improved to the equivalence of first-order stationary points. 
	\begin{theo}
		\label{The_equivalence_Ineq_FOSP}
		For any given $x \in \K$, suppose Assumption \ref{Assumption_Ineq} holds,  the penalty parameters $\beta$ and $\{\gamma_j\}$ in \ref{Prob_Pen} satisfy Condition \ref{Cond_beta_gamma}. If $y \in \BOmegax{x}$ is a first-order stationary point of \ref{Prob_Pen}, then $y$ is a first-order stationary point of \ref{Prob_Ori}. 
	\end{theo}

	Corollary \ref{Coro_equivalence_Ineq_KKT} illustrates that for any $x \in \K$, \ref{Prob_Ori} and \ref{Prob_Pen} have the same KKT points in $\BOmegax{x}$, where the lower bound for $\beta$ and $\{\gamma_j\}$ are independent of the multiplier  $\mu$. 
	\begin{coro}
		\label{Coro_equivalence_Ineq_KKT}
		For any given $x \in \K$, suppose Assumption \ref{Assumption_Ineq} holds,  the penalty parameters $\beta$ and $\{\gamma_j\}$ in \ref{Prob_Pen} satisfy Condition \ref{Cond_beta_gamma}. Then $y \in \BOmegax{x}$ is a KKT point of \ref{Prob_Pen} if and only if $y$ is a KKT point of \ref{Prob_Ori}. 
	\end{coro}

	The next corollary follows directly from Theorem \ref{The_fval_reduction}, hence  we omit its proof for simplicity.  
	
	\begin{coro}
		\label{Coro_fval_reduction_appendix}
		For any given $x \in \KA$, suppose Assumption \ref{Assumption_Ineq} holds,  the penalty parameters $\beta$ and $\{\gamma_j\}$ in \ref{Prob_Pen} satisfy Condition \ref{Cond_beta_gamma}. If $y \in \BOmegax{x}$ is a  KKT point of \ref{Prob_Pen} with multiplier  $\mu$, then it holds that 
		\begin{equation}
			\begin{aligned}
				\LCDP(y, \mu) \geq{}& \LCDP( \A(y), \mu ),\\
				\LCDP(y, \mu) \geq{}& \LCDP( \A^{\infty}(y), \mu ).
			\end{aligned}
		\end{equation}
	\end{coro}

	\section{Numerical Experiments}
	In this section, we present preliminary numerical results on applying various existing efficient solvers to \ref{Prob_Ori} through the constraint dissolving approach \ref{Prob_Pen}. 
	
	\subsection{Basic settings}
	All the numerical experiments are performed on a server with Dual Intel(R) Xeon(R) Gold 6242R CPU @ 3.10GHz$\times 20$ running Python 3.8 under Ubuntu 20.04. Moreover, we choose various existing solvers that are developed based on different methods, including the interior point solvers TRCON \cite{byrd1999interior,conn2000trust} and Ipopt \cite{wachter2006implementation}, the ALM-based solver ALGENCAN \cite{andreani2008augmented}, and the solvers PSQP (\url{http://www.cs.cas.cz/~luksan/subroutines.html}) and SLSQP \cite{kraft1988software} that are developed based on sequential programming methods. We choose the initial points to be the same for all the solvers, and terminate these solvers when the running time exceeds $1200$ seconds.   The implementations details for each solver are presented below.
	\begin{itemize}
		\item ALGENCAN (version 3.1.1): ALGENCAN is a FORTRAN package that employs GENCAN  as its subroutine.  We call the ALGENCAN solver through its Python interface provided by the pyOpt package \cite{pyopt-paper}. We set the ``epsopt = $10^{-6}$'' while keeping all the other parameters as their default values. 
		\item IPOPT (version 3.14.5): Ipopt solver is written in C++ and we run it through its python interface provided in CyIpopt package. We set ``hessian\_approximation = limited-memory'' to adopt the (L)BFGS method for an approximated Hessian. Moreover, we set ``tol = $10^{-6}$'' and keep all the other parameters as their default values.  
		\item PSQP: PSQP solver is programmed in FORTRAN and we run it through the Python interface provided by the pyOpt package. We set ``TOLG = $10^{-6}$'' and keep all the other parameters as their default values. 
		\item SLSQP: SLSQP is a FORTRAN solver provided by SciPy package, together with its Python interface integrated in SciPy-optimize package.  We set ``tol = $10^{-6}$'', ``ftol = $10^{-5}$'' and keep all the other parameters as their default values. 
		\item TRCON: The solver TRCON is developed purely in Python based on the trust-region method. We set ``hess = None'' to employ the BFGS strategy to compute approximated Hessians for the objective function, fix ``gtol = $10^{-6}$'' and keep all the other parameters as their default values. 
	\end{itemize}
	It is worth mentioning that under our settings, all the tested solvers only utilize the first-order derivatives of the objective function and constraints. Moreover, we employ the automatic differentiation package autograd (\url{https://github.com/HIPS/autograd}) to compute the derivatives automatically.  In the numerical results, we report the function value obtained, as well as the feasibility and stationarity of the result $\tilde{x}$ obtained by these test solvers. Here the feasibility is measured by $\norm{u(\tilde{x})} + \norm{c(\tilde{x})} + \norm{\max\{v(\tilde{x}), 0 \}}$, and the stationarity is measured by $\norm{\nabla_x \LNLP(\tilde{x}, \tilde{\rho},\tilde{\lambda}, \tilde{\mu} )}$, where $\tilde{\rho},\tilde{\lambda},  \tilde{\mu} = \mathop{\arg\min}\limits_{\rho \in \Rp,\lambda \in \bb{R}^{N_E},  \mu \in \bb{R}^{N_I}_+} \norm{ \LNLP(\tilde{x}, \rho,\lambda, \mu )}^2$.

	\subsection{Riemannian center of mass}
	In this subsection, we test the numerical performance of our constraint dissolving approach on Example \ref{Example_Rie_mass} by applying all the tested solvers to solve \ref{Prob_Pen}, and compare it with the baselines (i.e. directly applying the tested solvers to solve \ref{Prob_Ori}). In Example \ref{Example_Rie_mass}, we choose the submanifold $\M$ as the symplectic Stiefel manifold $\ca{S}_{m,q} := \{ X\in \bb{R}^{m\times q}: X\tp Q_m X = Q_q \}$. Moreover, we choose $N = 1000$ and randomly generate $s^*$ and $\{s_i\}_{1\leq i\leq N}$ on the symplectic manifold. For the Riemannian center of mass problem in Example \ref{Example_Rie_mass}, the corresponding \ref{Prob_Pen} can be expressed as 
	\begin{equation}
		\begin{aligned}
			\min_{x \in \bb{R}^n}\quad &\frac{1}{N} \sum_{i = 1}^N \norm{\A(x)- s_i}^2 + \frac{\beta}{2}\norm{c(x)}^2\\
			\text{s.t.} \quad &\norm{\A(x) - s^*}^2 \leq r,
		\end{aligned}
	\end{equation}
	where we fix $\beta = 1$. Moreover, we set the initial point as $s^*$  in each test instance for all the compared solvers.

	The numerical results are presented in Table \ref{Table_Stiefel_mass_n}. From that table, we can conclude that applying existing solvers to \ref{Prob_Pen} is significantly more efficient than directly apply those solvers to \ref{Prob_Ori}. For some large-scale instances, the improvement could be $10$-$40$ times for some solvers.  Moreover, we notice that SLSQP solver encounters errors due to the ill-conditioned LSQ subproblems when it is directly applied to \ref{Prob_Ori}. 
	Furthermore, our results report that ALGECAN is unstable for almost all the test instances when it is directly applied to \ref{Prob_Ori}.

	\begin{table}[!htbp]
		\begin{center}
			\footnotesize
			\begin{minipage}{\textwidth}
				\caption{A comparison between \ref{Prob_Ori} and \ref{Prob_Pen} on Riemannian center of mass problems.}
				\label{Table_Stiefel_mass_n}
				\begin{tabular*}{\textwidth}{c@{\extracolsep{\fill}}cccccccccc@{\extracolsep{\fill}}}
					\toprule \midrule
					\multicolumn{2}{c}{\multirow{2}{*}{Test problems}}& \multirow{2}{*}{Solvers} &
					\multicolumn{2}{@{}c@{}}{Function value} &\multicolumn{2}{@{}c@{}}{Substationarity} & \multicolumn{2}{@{}c@{}}{Feasibility} 
					& \multicolumn{2}{@{}c@{}}{CPU time (s)} \\
					\cmidrule{4-5}\cmidrule{6-7} \cmidrule{8-9} \cmidrule{10-11}%
					&& & \ref{Prob_Pen} & \ref{Prob_Ori} & \ref{Prob_Pen} & \ref{Prob_Ori} & \ref{Prob_Pen} & \ref{Prob_Ori} & \ref{Prob_Pen} & \ref{Prob_Ori} \\ \midrule
					\multirow{20}{*}{\begin{tabular}[l]{@{}c@{}}$(q,r)=$\\ $(10,0.01)$ \end{tabular}} &\multirow{5}{*}{$m=50$} 
					& ALGENCAN & 5.44e+00 & 5.44e+00 & 4.07e-08 & 3.75e-07 & 4.77e-11 & 1.20e-09 & 0.37 & 3.39 \\
					&& IPOPT & 5.44e+00 & 5.44e+00 & 9.95e-07 & 1.68e-07 & 1.00e-08 & 1.00e-06 & 0.61 & 9.77 \\
					&& PSQP & 5.44e+00 & 5.44e+00 & 7.01e-07 & 3.46e-07 & 7.42e-11 & 3.55e-12 & 0.42 & 1.49 \\
					&& SLSQP & 5.44e+00 & - & 3.80e-07 & - & 5.87e-06 & - & 3.68 & - \\
					&& TRCON & 5.44e+00 & 5.46e+00 & 5.99e-07 & 7.06e-04 & 2.16e-12 & 1.11e-07 & 0.59 & 30.39 \\ \cline{2-11} 
					&\multirow{5}{*}{$m=100$} 
					& ALGENCAN & 5.60e+00 & - & 3.95e-07 & - & 6.88e-11 & - & 1.03 & - \\
					&& IPOPT & 5.60e+00 & 5.60e+00 & 4.88e-07 & 2.14e-07 & 1.08e-08 & 1.00e-06 & 1.02 & 30.11 \\
					&& PSQP & 5.60e+00 & 5.60e+00 & 8.46e-07 & 2.11e-07 & 2.20e-09 & 4.10e-12 & 0.56 & 4.43 \\
					&& SLSQP & 5.60e+00 & - &  4.51e-07 & - & 8.22e-06 & - & 30.93 & - \\
					&& TRCON & 5.60e+00 & 5.74e+00 & 9.32e-07 & 5.81e-02 & 1.33e-12 & 6.45e-06 & 1.11 & 45.36 \\ \cline{2-11} 
					&\multirow{5}{*}{$m=150$}
					& ALGENCAN & 5.84e+00 & - & 5.71e-07 & - & 3.77e-09 & - & 1.13 & - \\
					&& IPOPT & 5.84e+00 & 5.84e+00 & 5.48e-07 & 3.51e-07 & 9.98e-09 & 1.00e-06 & 1.57 & 75.25 \\
					&& PSQP & 5.84e+00 & 5.84e+00 & 5.32e-07 & 7.18e-07 & 5.30e-11 & 4.18e-11 & 1.58 & 10.07 \\
					&& SLSQP & 5.84e+00 & - &  6.07e-07 & - & 6.21e-06 & - & 112.83 & - \\
					&& TRCON & 5.84e+00 & 6.01e+00 & 9.83e-07 & 6.73e-02 & 1.62e-12 & 1.85e-05 & 2.53 & 63.93 \\ \cline{2-11} 
					&\multirow{5}{*}{$m=200$}
					& ALGENCAN & 5.88e+00 & - & 4.69e-07 & - & 4.99e-11 & - & 1.47 & - \\
					&& IPOPT & 5.88e+00 & 5.88e+00 & 2.64e-07 & 3.87e-08 & 1.14e-08 & 1.00e-06 & 1.91 & 57.91 \\
					&& PSQP & 5.88e+00 & 5.88e+00 & 6.70e-07 & 4.08e-07 & 1.44e-10 & 6.97e-13 & 2.61 & 17.54 \\
					&& SLSQP & 5.88e+00 & - &  6.18e-07 & - & 7.40e-06 & - & 256.97 & - \\
					&& TRCON & 5.88e+00 & 6.06e+00 & 9.54e-07 & 7.12e-02 & 2.06e-11 & 2.32e-05 & 4.89 & 104.14 \\ \hline
					\multirow{20}{*}{\begin{tabular}[l]{@{}c@{}}$(m,r)=$\\ $(100,0.01)$ \end{tabular}} 
					&\multirow{5}{*}{$q=6$} 
					& ALGENCAN & 3.42e+00 & - & 1.72e-07 & - & 2.65e-09 & - & 0.60 & - \\
					&& IPOPT & 3.42e+00 & 3.42e+00 & 2.42e-07 & 1.06e-07 & 1.00e-08 & 1.00e-06 & 0.65 & 6.33 \\
					&& PSQP & 3.42e+00 & 3.42e+00 & 2.05e-07 & 6.85e-07 & 6.24e-10 & 9.60e-12 & 0.21 & 0.66 \\
					&& SLSQP & 3.42e+00 & - &  4.05e-04 & - & 7.90e-06 & - & 3.98 & - \\
					&& TRCON & 3.42e+00 & 3.42e+00 & 7.54e-07 & 7.46e-04 & 1.42e-12 & 1.11e-15 & 0.76 & 2.74 \\ \cline{2-11}
					&\multirow{5}{*}{$q=10$} 
					& ALGENCAN & 5.60e+00 & - & 3.95e-07 & - & 6.88e-11 & - & 1.03 & - \\
					&& IPOPT & 5.60e+00 & 5.60e+00 & 4.88e-07 & 2.14e-07 & 1.08e-08 & 1.00e-06 & 1.02 & 30.11 \\
					&& PSQP & 5.60e+00 & 5.60e+00 & 8.46e-07 & 2.11e-07 & 2.20e-09 & 4.10e-12 & 0.56 & 4.43 \\
					&& SLSQP & 5.60e+00 & - &  4.51e-07 & - & 8.22e-06 & - & 30.93 & - \\
					&& TRCON & 5.60e+00 & 5.74e+00 & 9.32e-07 & 5.81e-02 & 1.33e-12 & 6.45e-06 & 1.11 & 45.36 \\ \cline{2-11} 
					&\multirow{5}{*}{$q=14$} 
					& ALGENCAN & 8.08e+00 & - & 4.09e-07 & - & 1.90e-10 & - & 1.82 & - \\
					&& IPOPT & 8.08e+00 & 8.08e+00 & 3.42e-07 & 4.52e-08 & 1.13e-08 & 1.00e-06 & 1.92 & 56.56 \\
					&& PSQP & 8.08e+00 & 8.08e+00 & 7.33e-07 & 3.03e-07 & 3.18e-11 & 1.11e-11 & 1.80 & 18.54 \\
					&& SLSQP & 8.08e+00 & - &  4.12e-07 & - & 6.21e-06 & - & 110.97 & - \\
					&& TRCON & 8.08e+00 & 8.27e+00 & 4.87e-07 & 7.78e-02 & 1.98e-12 & 4.92e-06 & 2.21 & 114.87 \\ \cline{2-11} 
					&\multirow{5}{*}{$q=18$} 
					& ALGENCAN & 9.98e+00 & - & 3.76e-07 & - & 4.77e-10 & - & 2.04 & - \\
					&& IPOPT & 9.98e+00 & 9.98e+00 & 6.74e-07 & 3.63e-08 & 1.03e-08 & 1.00e-06 & 2.33 & 86.40 \\
					&& PSQP & 9.98e+00 & 9.98e+00 & 8.64e-07 & 2.73e-07 & 1.59e-10 & 8.89e-13 & 2.34 & 52.86 \\
					&& SLSQP & 9.98e+00 & - &  1.84e-07 & - & 6.28e-06 & - & 219.27 & - \\
					&& TRCON & 9.98e+00 & 1.02e+01 & 9.72e-07 & 7.83e-02 & 1.27e-11 & 4.48e-05 & 4.13 & 293.02 \\ \hline 
					\multirow{20}{*}{\begin{tabular}[l]{@{}c@{}}$(m,q)=$\\ $(100,10)$ \end{tabular}} 
					&\multirow{5}{*}{$r=0.001$} 
					& ALGENCAN & 5.78e+00 & - & 4.75e-08 & - & 1.13e-10 & - & 1.78 & - \\
					&& IPOPT & 5.78e+00 & 5.78e+00 & 9.99e-07 & 7.01e-10 & 1.04e-08 & 1.00e-05 & 1.68 & 38.36 \\
					&& PSQP & 5.78e+00 & 5.78e+00 & 9.90e-07 & 6.45e-07 & 2.81e-10 & 2.70e-12 & 1.62 & 8.04 \\
					&& SLSQP & 5.78e+00 & - &  5.52e-07 & - & 9.97e-06 & - & 56.10 & - \\
					&& TRCON & 5.78e+00 & 5.83e+00 & 6.83e-07 & 1.11e-01 & 3.97e-12 & 2.03e-05 & 6.65 & 45.06 \\ \cline{2-11}
					&\multirow{5}{*}{$r=0.01$}
					& ALGENCAN & 5.60e+00 & - & 3.95e-07 & - & 6.88e-11 & - & 1.03 & - \\
					&& IPOPT & 5.60e+00 & 5.60e+00 & 4.88e-07 & 2.14e-07 & 1.08e-08 & 1.00e-06 & 1.02 & 30.11 \\
					&& PSQP & 5.60e+00 & 5.60e+00 & 8.46e-07 & 2.11e-07 & 2.20e-09 & 4.10e-12 & 0.56 & 4.43 \\
					&& SLSQP & 5.60e+00 & - &  4.51e-07 & - & 8.22e-06 & - & 30.93 & - \\
					&& TRCON & 5.60e+00 & 5.74e+00 & 9.32e-07 & 5.81e-02 & 1.33e-12 & 6.45e-06 & 1.11 & 45.36 \\ \cline{2-11} 
					&\multirow{5}{*}{$r=0.1$}
					& ALGENCAN & 5.35e+00 & - & 1.10e-06 & - & 9.34e-10 & - & 0.52 & - \\
					&& IPOPT & 5.35e+00 & 5.35e+00 & 8.12e-07 & 7.60e-08 & 1.00e-08 & 1.00e-07 & 0.57 & 25.64 \\
					&& PSQP & 5.35e+00 & 5.35e+00 & 3.69e-07 & 3.18e-07 & 2.00e-11 & 2.06e-10 & 0.18 & 2.19 \\
					&& SLSQP & 5.35e+00 & - &  4.92e-05 & - & 8.65e-06 & - & 7.40 & - \\
					&& TRCON & 5.35e+00 & 5.51e+00 & 8.62e-07 & 6.38e-03 & 5.49e-12 & 1.23e-05 & 0.46 & 43.00 \\ \cline{2-11}
					&\multirow{5}{*}{$r=1$}
					& ALGENCAN & 4.85e+00 & - & 6.35e-07 & - & 3.45e-09 & - & 0.82 & - \\
					&& IPOPT & 4.85e+00 & 4.85e+00 & 8.07e-07 & 1.36e-07 & 9.96e-09 & 1.06e-08 & 0.94 & 49.84 \\
					&& PSQP & 4.85e+00 & 4.85e+00 & 9.32e-07 & 6.70e-07 & 5.93e-11 & 3.51e-10 & 0.20 & 2.67 \\
					&& SLSQP & 4.85e+00 & - &  4.57e-04 & - & 6.61e-06 & - & 6.01 & - \\
					&& TRCON & 4.85e+00 & 4.90e+00 & 6.62e-07 & 2.52e-02 & 8.24e-12 & 1.06e-11 & 0.46 & 11.87 \\ \hline
					\bottomrule
				\end{tabular*}
			\end{minipage}
		\end{center}
	\end{table}

	\subsection{Minimum balanced cut for graph bisection}
	In this subsection, we perform the numerical experiments on Example \ref{Example_Graph_cut}, where we choose the matrix $L$ as the Laplacian for a given graph $\ca{G}$, i.e. $L = D - A$, where $D\in \bb{R}^{m\times m}$ is the diagonal matrix of the degree of the vertices, and $A \in \bb{R}^{m\times m}$ is the adjacency matrix. In our numerical experiments, we randomly generate the  graph $\ca{G}$ by the Erd{\"o}s-R{\~e}nyi model through the build-in function $\ca{G} = \mathrm{Erdos\_Renyi}(m,\rho)$, where any two nodes in $\ca{G}$ are connected with probability $\rho$ independently. Furthermore, we randomly generate an initial point on the Oblique manifold. 
	In our proposed approaches, we transform \eqref{Example_BC} into the following optimization problem with $q$ equality constraints in $\bb{R}^{m\times q}$,
	\begin{equation}
		\label{Example_BC_CDP}
		\begin{aligned}
			\min_{X \in \bb{R}^{m\times q}} \quad & -\frac{1}{4} \tr\left( \A(X)\tp L\A(X) \right) + \frac{\beta}{4} \norm{\mathrm{Diag}(XX\tp ) - I_m}^2\\
			\text{s.t.}\quad &  \A(X)\tp e = 0,
		\end{aligned}
	\end{equation}
	where $\A(X) = 2X\left(\mathrm{Diag}(XX\tp) + I_q\right)^{-1} \in \bb{R}^{m\times q}$ and $\beta$ is fixed as $0.05$ for all test instances. 
	
	The numerical results are presented in Table \ref{Table_balanced_cut_n}. We observe that all the solvers successfully compute a solution for \ref{Prob_Ori} and \ref{Prob_Pen}. From Table \ref{Table_balanced_cut_n}, we can conclude that applying existing solvers to \ref{Prob_Pen} is significantly more efficient than directly applying those solvers to \ref{Prob_Ori}. Again, the improvement can be over $10$ times for some of the compared solvers. Moreover, among all the solvers, it appears that PSQP has the best overall performance. Therefore, we can conclude that \ref{Prob_Pen} enables the direct implementation of existing solvers while enjoying the high efficiency from utilizing the partial manifold  structures of \ref{Prob_Ori}. Furthermore, it is worth mentioning that constructing \ref{Prob_Pen} is independent of the geometrical properties of the Riemannian manifold $\M$. By solving \ref{Prob_Ori} via \ref{Prob_Pen}, we utilize the Riemannian structures of the constraints while avoiding the needs to deal with the geometrical materials of the Riemannian manifold $\M$.  
	
	\begin{table}[!htbp]
		\begin{center}
			\footnotesize
			\begin{minipage}{\textwidth}
				\caption{A comparison between \ref{Prob_Ori} and \ref{Prob_Pen} on minimum balanced cut problems.}
				\label{Table_balanced_cut_n}
				\begin{tabular*}{\textwidth}{c@{\extracolsep{\fill}}cccccccccc@{\extracolsep{\fill}}}
					\toprule \midrule
					\multicolumn{2}{c}{\multirow{2}{*}{Test problems}}& \multirow{2}{*}{Solvers} &
					\multicolumn{2}{@{}c@{}}{Function value} &\multicolumn{2}{@{}c@{}}{Substationarity} & \multicolumn{2}{@{}c@{}}{Feasibility} 
					& \multicolumn{2}{@{}c@{}}{CPU time (s)} \\
					\cmidrule{4-5}\cmidrule{6-7} \cmidrule{8-9} \cmidrule{10-11}%
					&& & \ref{Prob_Pen} & \ref{Prob_Ori} & \ref{Prob_Pen} & \ref{Prob_Ori} & \ref{Prob_Pen} & \ref{Prob_Ori} & \ref{Prob_Pen} & \ref{Prob_Ori} \\ \midrule
					\multirow{20}{*}{\begin{tabular}[l]{@{}c@{}}$(q,\rho)=$\\ $(2,0.1)$ \end{tabular}} 
					&\multirow{5}{*}{$m=50$} 
					& ALGENCAN & -1.62e+01 & -1.62e+01 & 4.05e-07 & 5.80e-07 & 3.12e-09 & 8.06e-09 & 0.85 & 2.21 \\
					&& IPOPT & -1.62e+01 & -1.62e+01 & 5.66e-07 & 5.42e-07 & 1.16e-10 & 3.52e-12 & 0.63 & 0.62 \\
					&& PSQP & -1.62e+01 & -1.62e+01 & 6.11e-07 & 5.22e-07 & 4.49e-11 & 5.39e-11 & 0.09 & 0.61 \\
					&& SLSQP & -1.62e+01 & -1.62e+01 & 4.74e-07 & 2.38e-07 & 4.08e-11 & 8.68e-13 & 0.16 & 0.72 \\
					&& TRCON & -1.62e+01 & -1.62e+01 & 9.41e-07 & 8.80e-06 & 2.24e-12 & 1.38e-10 & 0.70 & 6.02 \\ \cline{2-11} 
					&\multirow{5}{*}{$m=100$} 
					& ALGENCAN & -3.43e+01 & -3.43e+01 & 5.27e-07 & 6.25e-07 & 2.81e-09 & 3.51e-09 & 1.12 & 6.37 \\
					&& IPOPT & -3.43e+01 & -3.43e+01 & 9.88e-07 & 5.15e-07 & 1.95e-11 & 2.35e-11 & 0.79 & 7.73 \\
					&& PSQP & -3.43e+01 & -3.43e+01 & 7.71e-07 & 9.95e-07 & 4.44e-11 & 5.50e-11 & 0.17 & 2.16 \\
					&& SLSQP & -3.43e+01 & -3.43e+01 & 6.57e-07 & 1.23e-07 & 8.26e-11 & 3.42e-13 & 0.70 & 2.04 \\
					&& TRCON & -3.43e+01 & -3.43e+01 & 9.26e-07 & 9.92e-07 & 4.93e-12 & 8.01e-13 & 0.80 & 17.63 \\ \cline{2-11} 
					&\multirow{5}{*}{$m=200$} 
					& ALGENCAN & -7.44e+01 & -7.44e+01 & 1.64e-07 & 2.92e-07 & 2.09e-09 & 9.09e-10 & 5.84 & 34.68 \\
					&& IPOPT & -7.44e+01 & -7.44e+01 & 6.20e-07 & 8.92e-07 & 4.74e-10 & 5.48e-11 & 1.83 & 85.03 \\
					&& PSQP & -7.44e+01 & -7.44e+01 & 9.31e-07 & 9.34e-07 & 2.31e-10 & 1.51e-10 & 0.38 & 24.35 \\
					&& SLSQP & -7.44e+01 & -7.44e+01 & 9.89e-07 & 1.21e-07 & 3.30e-11 & 5.18e-13 & 7.45 & 23.67 \\
					&& TRCON & -7.44e+01 & -7.42e+01 & 9.96e-07 & 1.65e-02 & 1.17e-11 & 3.18e-04 & 2.21 & 49.28 \\ \cline{2-11} 
					&\multirow{5}{*}{$m=400$} 
					& ALGENCAN & -1.63e+02 & - & 7.15e-07 & - & 1.37e-10 & - & 60.40 & - \\
					&& IPOPT & -1.63e+02 & -1.63e+02 & 9.17e-07 & 9.36e-07 & 7.83e-11 & 2.82e-10 & 8.41 & 739.39 \\
					&& PSQP & -1.63e+02 & -1.63e+02 & 8.14e-07 & 7.08e-07 & 2.12e-10 & 1.22e-10 & 1.19 & 251.36 \\
					&& SLSQP & -1.63e+02 & -1.63e+02 & 5.22e-07 & 7.00e-08 & 4.45e-11 & 2.68e-13 & 75.74 & 181.46 \\
					&& TRCON & -1.63e+02 & -1.62e+02 & 7.81e-05 & 6.09e-03 & 9.61e-08 & 1.86e-05 & 6.92 & 136.06 \\  \hline 
					\multirow{20}{*}{\begin{tabular}[l]{@{}c@{}}$(m,\rho)=$\\ $(100,0.1)$ \end{tabular}} 
					&\multirow{4}{*}{$q=2$} 
					& ALGENCAN & -3.43e+01 & -3.43e+01 & 5.27e-07 & 6.25e-07 & 2.81e-09 & 3.51e-09 & 1.12 & 6.37 \\
					&& IPOPT & -3.43e+01 & -3.43e+01 & 9.88e-07 & 5.15e-07 & 1.95e-11 & 2.35e-11 & 0.79 & 7.73 \\
					&& PSQP & -3.43e+01 & -3.43e+01 & 7.71e-07 & 9.95e-07 & 4.44e-11 & 5.50e-11 & 0.17 & 2.16 \\
					&& SLSQP & -3.43e+01 & -3.43e+01 & 6.57e-07 & 1.23e-07 & 8.26e-11 & 3.42e-13 & 0.70 & 2.04 \\
					&& TRCON & -3.43e+01 & -3.43e+01 & 9.26e-07 & 9.92e-07 & 4.93e-12 & 8.01e-13 & 0.80 & 17.63 \\ \cline{2-11} 
					&\multirow{4}{*}{$q=4$} 
					& ALGENCAN & -3.60e+01 & -3.60e+01 & 3.52e-07 & 2.61e-07 & 8.78e-09 & 3.22e-09 & 1.24 & 10.57 \\
					&& IPOPT & -3.60e+01 & -3.60e+01 & 7.42e-07 & 3.82e-07 & 1.42e-11 & 7.74e-12 & 1.27 & 40.78 \\
					&& PSQP & -3.60e+01 & -3.60e+01 & 1.07e-06 & 7.60e-07 & 9.36e-11 & 1.04e-10 & 0.34 & 5.54 \\
					&& SLSQP & -3.60e+01 & -3.60e+01 & 8.38e-07 & 1.20e-07 & 1.83e-12 & 9.12e-13 & 4.06 & 7.64 \\
					&& TRCON & -3.60e+01 & -3.60e+01 & 9.09e-07 & 1.78e-04 & 1.75e-10 & 2.83e-08 & 1.57 & 24.09 \\ \cline{2-11} 
					&\multirow{4}{*}{$q=6$} 
					& ALGENCAN & -3.81e+01 & -3.81e+01 & 2.50e-07 & 1.79e-07 & 6.34e-10 & 2.46e-09 & 1.90 & 8.76 \\
					&& IPOPT & -3.81e+01 & -3.81e+01 & 5.76e-07 & 7.24e-07 & 2.45e-10 & 1.83e-10 & 1.83 & 83.78 \\
					&& PSQP & -3.81e+01 & -3.81e+01 & 8.55e-07 & 9.49e-07 & 9.75e-11 & 1.47e-10 & 0.91 & 10.96 \\
					&& SLSQP & -3.81e+01 & -3.81e+01 & 6.33e-07 & 7.37e-08 & 2.42e-11 & 5.82e-13 & 14.35 & 22.08 \\
					&& TRCON & -3.81e+01 & -3.81e+01 & 3.85e-04 & 4.02e-05 & 2.38e-08 & 1.47e-09 & 6.54 & 30.70 \\ \cline{2-11} 
					&\multirow{4}{*}{$q=8$} 
					& ALGENCAN & -4.10e+01 & -4.10e+01 & 7.32e-08 & 8.80e-08 & 4.17e-09 & 2.07e-09 & 1.55 & 10.11 \\
					&& IPOPT & -4.10e+01 & -4.10e+01 & 8.89e-07 & 6.54e-07 & 2.87e-11 & 7.69e-12 & 1.93 & 111.52 \\
					&& PSQP & -4.10e+01 & -4.10e+01 & 9.74e-07 & 8.92e-07 & 6.41e-11 & 8.06e-13 & 1.33 & 13.38 \\
					&& SLSQP & -4.10e+01 & -4.10e+01 & 8.13e-07 & 7.13e-08 & 2.61e-11 & 3.78e-13 & 25.97 & 41.20 \\
					&& TRCON & -4.10e+01 & -4.10e+01 & 2.63e-05 & 1.49e-05 & 5.83e-11 & 1.75e-10 & 7.59 & 37.19 \\ \hline
					\multirow{16}{*}{\begin{tabular}[l]{@{}c@{}}$(m,q)=$\\ $(100,2)$ \end{tabular}} 
					&\multirow{4}{*}{$\rho=0.05$} 
					& ALGENCAN & -3.28e+01 & -3.30e+01 & 7.08e-07 & 8.31e-08 & 1.91e-09 & 2.44e-09 & 0.99 & 7.99 \\
					&& IPOPT & -3.28e+01 & -3.27e+01 & 8.61e-07 & 6.02e-07 & 1.83e-09 & 1.13e-11 & 1.77 & 11.97 \\
					&& PSQP & -3.28e+01 & -3.30e+01 & 8.63e-07 & 8.36e-07 & 9.60e-11 & 1.08e-10 & 0.26 & 2.82 \\
					&& SLSQP & -3.28e+01 & -3.30e+01 & 9.32e-07 & 1.34e-07 & 2.84e-12 & 4.58e-13 & 1.08 & 2.33 \\
					&& TRCON & -3.28e+01 & -3.30e+01 & 9.88e-07 & 3.30e-05 & 4.22e-12 & 1.80e-09 & 1.44 & 19.61 \\ \cline{2-11} 
					&\multirow{4}{*}{$\rho=0.1$} 
					& ALGENCAN & -3.43e+01 & -3.43e+01 & 5.27e-07 & 6.25e-07 & 2.81e-09 & 3.51e-09 & 1.12 & 6.37 \\
					&& IPOPT & -3.43e+01 & -3.43e+01 & 9.88e-07 & 5.15e-07 & 1.95e-11 & 2.35e-11 & 0.79 & 7.73 \\
					&& PSQP & -3.43e+01 & -3.43e+01 & 7.71e-07 & 9.95e-07 & 4.44e-11 & 5.50e-11 & 0.17 & 2.16 \\
					&& SLSQP & -3.43e+01 & -3.43e+01 & 6.57e-07 & 1.23e-07 & 8.26e-11 & 3.42e-13 & 0.70 & 2.04 \\
					&& TRCON & -3.43e+01 & -3.43e+01 & 9.26e-07 & 9.92e-07 & 4.93e-12 & 8.01e-13 & 0.80 & 17.63 \\ \cline{2-11} 
					&\multirow{4}{*}{$\rho=0.2$} 
					& ALGENCAN & -4.21e+01 & -4.21e+01 & 9.60e-08 & 3.21e-07 & 1.22e-09 & 5.36e-09 & 1.68 & 7.66 \\
					&& IPOPT & -4.21e+01 & -4.21e+01 & 9.95e-07 & 9.88e-07 & 1.67e-09 & 1.55e-10 & 0.73 & 8.55 \\
					&& PSQP & -4.21e+01 & -4.21e+01 & 6.60e-07 & 7.11e-07 & 6.00e-11 & 2.82e-10 & 0.16 & 3.22 \\
					&& SLSQP & -4.21e+01 & -4.21e+01 & 9.57e-07 & 1.21e-07 & 8.63e-12 & 6.58e-13 & 0.61 & 2.75 \\
					&& TRCON & -4.21e+01 & -4.20e+01 & 9.98e-07 & 2.18e-02 & 3.60e-12 & 2.89e-04 & 1.27 & 19.15 \\ \cline{2-11} 
					&\multirow{4}{*}{$\rho=0.3$} 
					& ALGENCAN & -4.37e+01 & -4.37e+01 & 2.52e-07 & 1.77e-08 & 1.46e-09 & 1.98e-10 & 7.51 & 10.21 \\
					&& IPOPT & -4.37e+01 & -4.37e+01 & 3.61e-07 & 4.90e-07 & 1.46e-11 & 4.18e-11 & 1.83 & 10.92 \\
					&& PSQP & -4.37e+01 & -4.37e+01 & 9.88e-07 & 6.41e-07 & 6.01e-11 & 3.81e-11 & 0.21 & 4.78 \\
					&& SLSQP & -4.37e+01 & -4.37e+01 & 8.90e-07 & 1.30e-07 & 7.63e-11 & 6.61e-13 & 0.97 & 3.99 \\
					&& TRCON & -4.37e+01 & -4.37e+01 & 6.80e-03 & 1.95e-03 & 4.21e-04 & 1.68e-06 & 3.87 & 19.91 \\ \hline
					\bottomrule
				\end{tabular*}
			\end{minipage}
		\end{center}
	\end{table}

	\section{Conclusion}
	
	In this paper, we focus on a general constrained optimization problem \ref{Prob_Ori}, where some of its equality constraints define an embedded submanifold in $\Rn$. We propose a constraint dissolving approach for \ref{Prob_Ori} to achieve better convenience and efficiency by directly applying various of Euclidean optimization solvers while exploiting the Riemannian structure inside \eqref{Prob_Ori}. In our proposed constraint dissolving approaches,  solving \ref{Prob_Ori} is transformed into solving the constraint dissolving problem \ref{Prob_Pen}, which is a constrained optimization problem in $\Rn$ that has eliminated the manifold constraints $c(x) = 0$ from \ref{Prob_Ori}.  Under mild conditions, we prove the equivalence between \ref{Prob_Ori} and \ref{Prob_Pen}, in the aspect of their constraint qualifications, stationary points and KKT points. Moreover, for the special case with no additional equality constraints except $c(x) = 0$ for \ref{Prob_Ori}, we prove sharper results on the equivalence between \ref{Prob_Ori} and \ref{Prob_Pen}.  Therefore, we can directly apply existing Euclidean optimization approaches to solve \ref{Prob_Pen} and enjoy the rich expertise gained over the past years for solving  constrained optimization in $\Rn$. 
	
	We perform preliminary numerical experiments on Riemannian center of mass problems and balanced graph bisection problems, where various existing numerical solvers, including ALGENCAN, IPOPT, PSQP, SLSQP, and TRCON are applied to solve \ref{Prob_Ori} and \ref{Prob_Pen}. The numerical results demonstrate that applying these solvers to \ref{Prob_Pen} can achieve superior performance in efficiency and stability over the approaches of directly solving \ref{Prob_Ori}.   Therefore, we can conclude that our constraint dissolving approach enjoys the great convenience in directly applying various Euclidean optimization solvers, and its numerical performance can be boosted by exploiting the structures of $\M$ in \ref{Prob_Ori}.

	\bibliography{ref}
	\bibliographystyle{plainnat}
	\appendix

	\section{Proofs for Theoretical Results}
	\subsection{Preliminary lemmas}
	In this subsection, we present several preliminary lemmas from  \cite[Section 3.1]{xiao2022dissolving}.  
	\begin{lem}
		\label{Le_bound_cx}
		For any given $x \in \M$, the following inequalities hold for any $\y \in \Omegax{x}$,
		\begin{equation}
			\frac{1}{\Mc } \norm{c(\y)} \leq \mathrm{dist}(\y, \M) \leq \frac{2}{\Lsc } \norm{c(\y)}.
		\end{equation}
	\end{lem}

	\begin{lem}
		\label{Le_Ax_c}
		For any given $x \in \M$, it holds that
		\begin{equation}
			\norm{\A(\y) - \y }  \leq \frac{2(\Ma +1)}{\Lsc }\norm{c(\y)}, \qquad \text{ for any $\y \in \Omegax{x}$}.
		\end{equation}
	\end{lem}

	\begin{lem}
		\label{Le_A_secondorder_descrease}
		For any given $x \in \M$,  it holds that 
		\begin{equation}
			\norm{c(\A(\y))} \leq \frac{4\Lac }{\Lsc ^2}\norm{c(\y)}^2, \qquad \text{for any $\y \in \Omegax{x}$}.
		\end{equation}
	\end{lem}

	\begin{lem}
		\label{Le_Ja_identical}
		For any given $x \in \M$, the inclusion $\Ja(x)\tp d  \in \Tx$ holds for any $d  \in \bb{R}^n$. 
		Moreover, when $d \in \Tx$, it holds that $\Ja(x)\tp d = d$. 
		
	\end{lem}

	\begin{lem}
		\label{Le_Ja_nullspace}
		Given any $x \in \M$, $\Ja(x)d = 0$ if and only if $d \in \Nx = \mathrm{range}(\Jc(x))$.
	\end{lem}

	\begin{lem}
		\label{Le_Ja_idempotent}
		Given any $x \in \M$, it holds that $\Ja(x)^2 = \Ja(x)$. 
	\end{lem}

	\subsection{Proofs for Section 3.1}
	\label{Subsection_proofs_CQs}
	\begin{lem}
		\label{Le_subset_ACQ}
		For any $x \in \K$, it holds that 
		\begin{align}
			&\TK(x) \subseteq \TKA(x),\\
			& \TKlin (x) \subseteq \TKAlin(x). 
		\end{align}
	\end{lem}
	This result can be directly derived from the fact that $\K \subseteq \KA$. 
	
	\begin{lem}
		\label{Le_subspace_polar_cone}
		For any closed cone $\ca{X}_1 \subset \ca{R}^n$, let  $\ca{X}_2 = \Ja(x)\tp \ca{X}_1$, then 
		\begin{equation}
			\ca{X}_1^\circ \cap \mathrm{range}(\Ja(x)) = \Ja(x) \ca{X}_2^\circ. 
		\end{equation}
	\end{lem}
	\begin{proof}
		For any $d_3 \in \Ja(x) \ca{X}_2^\circ$, there exists $d_2 \in \ca{X}_2^{\circ}$ such that $d_3 = \Ja(x) d_2$. As a result, from the definition of the polar cone, it holds for any $d_1 \in \ca{X}_1$ that 
		\begin{equation}
			0\geq \inner{d_2, \Ja(x)\tp d_1} =\inner{ \Ja(x)d_2, d_1} = \inner{d_3, d_1}. 
		\end{equation}
		Therefore, from the arbitrariness of $d_1\in \X_1$, we can conclude that  $d_3 \in \ca{X}_1^{\circ}$ and hence $\Ja(x)\ca{X}_2^\circ \subseteq \ca{X}_1^\circ$. Together with the fact that $\Ja(x) \ca{X}_2^\circ \subset \mathrm{range}(\Ja(x))$, we have that 
		\begin{equation}
			\label{Eq_Le_subspace_polar_cone_0}
			\Ja(x) \ca{X}_2^\circ \subseteq \ca{X}_1^\circ \cap \mathrm{range}(\Ja(x)). 
		\end{equation}
		
		On the other hand, for any $d_3 \in \ca{X}_1^\circ \cap \mathrm{range}(\Ja(x))$, $\inner{d_3, d_1}\leq 0$ holds for  any $d_1 \in \ca{X}_1$. Notice that $d_3 \in \mathrm{range}(\Ja(x))$, we have that $\Ja(x){\Ja(x)}^\dagger d_3 = d_3$. Therefore,  for  any $d_1 \in \ca{X}_1$, we have
		\begin{equation}
			0\geq \inner{d_3, d_1} = \inner{\Ja(x){\Ja(x)}^\dagger d_3, d_1} = \inner{{\Ja(x)}^\dagger d_3, {\Ja(x)}\tp d_1},
		\end{equation}
		which illustrates that ${\Ja(x)}^\dagger d_3 \in \ca{X}_2^{\circ}$. Therefore,  we conclude that 
		\begin{equation}
			d_3 = \Ja(x){\Ja(x)}^\dagger d_3 \in \Ja(x)\ca{X}_2^{\circ}.
		\end{equation}
		Therefore, together with \eqref{Eq_Le_subspace_polar_cone_0}, we  achieve $ \ca{X}_1^\circ \cap \mathrm{range}(\Ja(x)) =  \Ja(x) \ca{X}_2^\circ $. This completes the proof. 
	\end{proof}

	\begin{lem}
		\label{Le_TK_TM}
		For any $x \in \K$, it holds that $\TK(x) \subseteq \Tx$. 
	\end{lem}
	\begin{proof} 
		For any $d \in \TK(x)$, there exists a sequence $\{\xk\} \subset \K \subset \M$ and $\{t_k\} \to 0$ such that  $d = \lim_{k \to +\infty} \frac{x_k - x}{t_k}$. From the definition of $\Tx$, we can conclude that $d \in \Tx$, hence complete the proof. 
	\end{proof}

	\paragraph{Proof for Proposition \ref{Prop_Equivalence_LICQ}}
	\begin{proof}
		
		For any coefficients $\hat{\lambda} \in \bb{R}^{N_E}$,  and $\hat{\mu} \in \bb{R}^{N_I}$  such that 
		\begin{equation}
			0 = \sum_{i \in [N_E]} \hat{\lambda}_i \nabla \tilde{\ec}_i(x) + \sum_{j \in \ca{F}_A(x)} \hat{\mu}_j \nabla \tilde{\ic}_j(x),
		\end{equation}
		it holds that
		\begin{equation}
			\begin{aligned}
				0 ={}& \sum_{i \in [N_E]} \hat{\lambda}_i \nabla \tilde{\ec}_i(x) + \sum_{j \in \ca{F}_A(x)} \hat{\mu}_j \nabla \tilde{\ic}_j(x)\\
				={}& \Ja(x) \left( \sum_{i \in [N_E]} \hat{\lambda}_i \nabla \ec_i(x) + \sum_{j \in \ca{F}_A(x)} \hat{\mu}_j \nabla \ic_j(x)   \right).
			\end{aligned}
		\end{equation}
		Notice that for any $x \in \M$, we have $\tilde{\ic}(x) = \ic(x)$, hence $\ca{F}_A(x) = \ca{F}(x)$ holds for any $x \in \M$. Then it holds from Lemma \ref{Le_Ja_nullspace}  that 
		\begin{equation}
			\sum_{i \in [N_E]} \hat{\lambda}_i \nabla \ec_i(x) + \sum_{j \in \ca{F}(x)} \hat{\mu}_j \nabla \ic_j(x) \in \mathrm{range}(\Jc(x)). 
		\end{equation}
		That is, there exists $\hat{\rho} \in \Rp$ such that 
		\begin{equation}
			\sum_{l \in [p]} \hat{\rho}_l\nabla  c_l(x) + \sum_{i \in [N_E]} \hat{\lambda}_i \nabla \ec_i(x) +  \sum_{j \in \ca{F}(x)} \hat{\mu}_j \nabla \ic_j(x) = 0. 
		\end{equation}
		Since LICQ holds at $x$ with respect to \ref{Prob_Ori}, we have $\{\nabla u_i(x): i\in [N_E]\}\cup\{\nabla c_l(x): l\in [p]\}\cup\{\nabla v_j(x): j\in \ca{F}(x)\}$ is linearly independent set, hence we obtain that $\hat{\rho} = 0$,  $\hat{\lambda} = 0$ and $\hat{\mu} = 0$. Therefore, we conclude that the LICQ with respect to \ref{Prob_Pen} holds at $x$.

	\end{proof}

	\paragraph{Proof for Proposition \ref{Prop_Equivalence_MFCQ}}
	\begin{proof}
		From the definition of MFCQ with respect to \ref{Prob_Ori}, it holds that $\{\nabla u_i(x): i\in [N_E]\}\cup\{\nabla c_l(x): l\in [p]\}$  is a linearly independent set and there exists $d \in \bb{R}^n$ such that 
		\begin{equation}
			\label{Eq_Prop_Equivalence_MFCQ_0}
			\begin{aligned}
				&\inner{d, \nabla v_j(x)} \leq -1, ~\forall j \in \ca{F}(x);\\
		 		&\inner{d, \nabla u_i(x)} = 0, ~\forall i \in [N_E];\\
		 		&\inner{d, \nabla c_i(x)} = 0, ~\forall i \in [p].
			\end{aligned}
		\end{equation}
		Therefore, for any $\hat{\lambda} \in \bb{R}^{N_E}$ such that $\sum_{i \in [N_E]} \hat{\lambda}_i \Ja(x) \nabla u_i(x) = 0$, Lemma \ref{Le_Ja_nullspace} implies that 
		\begin{equation*}
			\sum_{i \in [N_E]} \hat{\lambda}_i \nabla u_i(x) \in \mathrm{range}(\Jc(x)).
		\end{equation*}
		Then by the linear independence of $\{\nabla u_i(x): i\in [N_E]\}\cup\{\nabla c_l(x): l\in [p]\}$ , we get $\hat{\lambda} = 0$. As a result, we can conclude that  $\{ \Ja(x) \nabla u_i(x): i \in [N_E] \}$ is a linearly independent set. 
		Moreover, \eqref{Eq_Prop_Equivalence_MFCQ_0} illustrates that $d \in \Tx$. Combining with Lemma \ref{Le_Ja_identical}, it holds for any $j \in [N_I]$ that 
		\begin{equation}
			\inner{ d, \Ja(x) \nabla v_j(x)} = \inner{ \Ja(x)\tp d,  \nabla v_j(x)} =\inner{d, \nabla v_j(x) } \leq -1.
		\end{equation}
		Therefore, we conclude that MFCQ with respect to \ref{Prob_Pen} holds at $x$. 
	\end{proof}

	\paragraph{Proof for Lemma \ref{Le_relationship_ACQ}}
	\begin{proof}
		{\bf Proof for \eqref{Eq_Le_relationship_ACQ_0}: }
		For any $d_1 \in \TKAlin(x)$, it holds from the definition of $\TKAlin(x)$ that 
		\begin{equation}
			\inner{d_1, \Ja(x) \nabla u_i(x)} = 0, ~\inner{d_1, \Ja(x) \nabla  v_j(x)} \leq 0, ~\forall ~i \in [N_E], ~j\in \ca{F}_A(x).
		\end{equation}
		Notice that $\ca{F}_A(x) = \ca{F}(x)$ holds for any $x \in \K$. Then for any $i \in [N_E], j\in \ca{F}(x)$, we obtain
		\begin{equation}
			\inner{\Ja(x)\tp d_1, \nabla u_i(x)} = 0, ~ \text{and}~ \inner{\Ja(x)\tp d_1, \nabla  v_j(x)} \leq 0. 
		\end{equation}
		Since $x\in\K\subset\ca{M}$, Assumption~\ref{Assumption_2} implies
		that  $\inner{\Ja(x)\tp d_1, \nabla c_l(x)} = 0$ holds for any $l\in[p]$.
		Therefore, based on Definition \ref{Defin_linearize_cones}, it holds that  $\Ja(x)\tp d_1 \in \TKlin(x)$. From arbitrariness of $d_1 \in \TKAlin(x)$, we obtain that $\Ja(x)\tp \TKAlin(x) \subseteq \TKlin(x)$. 
		
		On the other hand,  Lemma \ref{Le_subset_ACQ} illustrates that $\TKlin(x) \subseteq \TKAlin(x)$. Notice that $\TKlin(x) \subset \Tx$, then it holds from Lemma \ref{Le_Ja_identical} that 
		\begin{equation}
			\TKlin(x) = \Ja(x)\tp \TKlin(x) \subseteq \Ja(x)\tp \TKAlin(x).  
		\end{equation}
		As a result, it holds that 
		\begin{equation}
			\Ja(x)\tp \TKAlin(x) = \TKlin(x).
		\end{equation}
		Together with Lemma \ref{Le_subspace_polar_cone}, it holds that 
		\begin{equation}
			\TKAlin(x)^\circ \cap \mathrm{range}(\Ja(x)) = \Ja(x)\TKlin(x)^\circ. 
		\end{equation}

		{\bf Proof for \eqref{Eq_Le_relationship_ACQ_1}: }
		Lemma \ref{Le_subset_ACQ} illustrates that $\TK(x) \subseteq \TKA(x)$. Moreover, from Lemma \ref{Le_TK_TM}, it holds that  $\TK(x) \subset \Tx$. Therefore, from Lemma \ref{Le_Ja_identical} we have that 
		\begin{equation}
			\TK(x) = \Ja(x)\tp \TK(x) \subseteq \Ja(x)\tp \TKA(x).  
		\end{equation}
		Hence $\big(\Ja(x)\tp \TKA(x)\big)^\circ \subset \TK(x)^\circ.$
		Together with Lemma \ref{Le_subspace_polar_cone}, it holds that 
		\begin{equation}
			\TKA(x)^\circ \cap \mathrm{range}(\Ja(x)) = \Ja(x)\left(\Ja(x)\tp \TKA(x)\right)^\circ \subseteq \Ja(x)\TK(x)^\circ. 
		\end{equation}
		This completes the proof. 
	\end{proof}


	\paragraph{Proof for Proposition \ref{Prop_CQ_equivalence_GCQ}}
	\begin{proof}
		From Lemma \ref{Le_relationship_ACQ}, it holds that 
		\begin{equation}
			\begin{aligned}
				&\TKA(x)^\circ \cap \mathrm{range}(\Ja(x)) \subseteq \Ja(x)\TK(x)^\circ \\
				={}& \Ja(x)\TKlin(x)^\circ = \TKAlin(x)^\circ \cap \mathrm{range}(\Ja(x))\\
				\subseteq{}& \TKA(x)^\circ \cap \mathrm{range}(\Ja(x)).
			\end{aligned}
		\end{equation}
		Here the first inclusion follows from \eqref{Eq_Le_relationship_ACQ_1}, the second equality is implied by \eqref{Eq_Le_relationship_ACQ_0}, and  the last inclusion uses the fact that  $\TKAlin(x)^\circ \subseteq \TKA(x)^\circ$ holds for any $x \in \KA$. 
		Therefore, we obtain that 
		\begin{equation}
			\TKA(x)^\circ \cap \mathrm{range}(\Ja(x)) = \TKAlin(x)^\circ \cap \mathrm{range}(\Ja(x)),
		\end{equation}
		and complete the proof. 
	\end{proof}

	\paragraph{Proof for Theorem \ref{The_CQ_fin}}
	\begin{proof}
		From Definition \ref{Defin_tangent_cones} and Definition \ref{Defin_linearize_cones}, it is easy to verify that $\TKAlin(x)^\circ \subseteq \TKA(x)^\circ$  holds for any $x \in \KA$. Therefore, for any $x \in \K \subseteq \KA$ that is a KKT point of \ref{Prob_Pen}, $x$ is also a first-order stationary point of \ref{Prob_Pen}. 
		 
		On the other hand, for any $x \in \K$ that is a first-order stationary point of \ref{Prob_Pen}, it holds from Definition \ref{Defin_FOSP} that 
		\begin{equation*}
			0 \in \partial h(x) + \TKA(x)^\circ = \Ja(x)\partial f(x) + \TKA(x)^\circ. 
		\end{equation*}
		Notice that $\Ja(x)\partial f(x) \subset \mathrm{range}(\Ja(x))$. Then it holds that 
		\begin{equation*}
			\begin{aligned}
				0 \in{}& \left(\Ja(x)\partial f(x) + \TKA(x)^\circ\right) \cap \mathrm{range}(\Ja(x)) \\
				={}& \Ja(x)\partial f(x) + \left( \TKA(x)^\circ \cap\mathrm{range}(\Ja(x))  \right)\\
				={}& \Ja(x)\partial f(x) + \left( \TKAlin(x)^\circ \cap\mathrm{range}(\Ja(x))  \right)\\
				\subseteq{}& \Ja(x)\partial f(x) + \TKAlin(x)^\circ = \partial h(x) +  \TKAlin(x)^\circ,
			\end{aligned}
		\end{equation*}
		where the last equality follows from  Proposition \ref{Prop_partial_h}. 
		Therefore, we can conclude that $x$ is a KKT point of \ref{Prob_Pen} and complete the proof. 
	\end{proof}

	\subsection{Proofs for Subsection 3.2}
	\label{Subsection_proofs_equivalence}

	\paragraph{Proof for Theorem \ref{The_equivalence_feasible}}
	\begin{proof}
		{\bf Proof for Theorem \ref{The_equivalence_feasible}(1)}

		When $x$ is a first-order stationary point of \ref{Prob_Pen},  $x \in \K$ implies that $c(x) = 0$. Then it holds that 
		\begin{equation*}
			0 \in \partial h(x) + \TKA(x)^\circ = \Ja(x) \partial f(x) +  \TKA(x)^\circ. 
		\end{equation*} 
		Then it holds from \eqref{Eq_Le_relationship_ACQ_1} in  Lemma \ref{Le_relationship_ACQ} that
		\begin{equation*}
			\begin{aligned}
				&0 \in \left( \Ja(x) \partial f(x) +  \TKA(x)^\circ \right) \cap \mathrm{range}(\Ja(x))\\
				\subseteq{}& \Ja(x) \partial f(x)  + \left( \TKA(x)^\circ  \cap \mathrm{range}(\Ja(x)) \right)\\
				\subseteq{}& \Ja(x) \partial f(x) + \Ja(x) \TK(x)^\circ = \Ja(x) (\partial f(x) + \TK(x)^\circ).
			\end{aligned}
		\end{equation*}
		As a result, from Lemma \ref{Le_Ja_nullspace}, there exists $\hat{\rho} \in \Rp$ such that $0 \in \partial f(x) + \Jc(x)\hat{\rho} + \TK(x)^\circ $. Additionally,  Lemma \ref{Le_TK_TM} illustrates that $\Nx \subseteq \TK^{\circ}$, which leads to the fact that $\Jc(x)\hat{\rho} + \TK(x)^\circ = \TK(x)^\circ$. 
		Therefore, together with Definition \ref{Defin_FOSP}, we obtain that 
		\begin{equation*}
			0 \in \partial f(x) + \TK(x)^\circ,
		\end{equation*}
		hence $x$ is a first-order stationary point of \ref{Prob_Ori}.

		{\bf Proof for Theorem \ref{The_equivalence_feasible}(2)}
		
		Suppose $x \in \K$ is a KKT point to \ref{Prob_Ori}, then from Definition \ref{Defin_KKT_NLP}, there exists $\rho \in \Rp$, $\lambda \in \bb{R}^{N_E}$ and $\mu \in \bb{R}^{N_I}_+$ such that 
		\begin{equation*}
			0 \in \partial f(x) +\sum_{l \in[p]} \rho_l \nabla c_l(x) +  \sum_{i \in[N_E]} \lambda_i \nabla u_i(x)   + \sum_{j \in \ca{F}(x)} \mu_j \nabla v_j(x).
		\end{equation*} 
		Lemma \ref{Le_Ja_identical} implies that  $\Ja(x) \nabla c_l(x) = 0$ holds for any $l \in [p]$. Therefore, we obtain
		\begin{equation*}
			\begin{aligned}
				&0 \in \Ja(x)\Big( \partial f(x) + \sum_{l \in[p]} \rho_l \nabla c_l(x)  + \sum_{i \in[N_E]} \lambda_i \nabla u_i(x) +   \sum_{j \in \ca{F}(x)} \mu_j \nabla v_j(x) \Big)\\
				={}& \partial h(x) + \sum_{i \in[N_E]} \lambda_i \Ja(x) \nabla u_i(x) + \sum_{j \in \ca{F}(x)} \mu_j  \Ja(x) \nabla  v_j(x)\\
				={}& \partial h(x) + \sum_{i \in[N_E]} \lambda_i  \nabla \tilde{\ec}(x) + \sum_{j \in \ca{F}(x)} \mu_j  \nabla \tilde{\ic}(x),
			\end{aligned}
		\end{equation*}
		which shows that $x$ is a KKT point to \ref{Prob_Pen}. 
		
		On the other hand, when $x \in \K$ is a KKT point to \ref{Prob_Pen}, it holds that 
		\begin{equation*}
			\begin{aligned}
				0\in{}&  \partial h(x) + \sum_{i \in[N_E]} \lambda_i \nabla  \tilde{\ec}_i(x) + \sum_{j \in \ca{F}(x)} \mu_j \nabla \tilde{\ic}_j(x)\\
				={}& \Ja(x)\Big( \partial f(x) + \sum_{i \in[N_E]} \lambda_i \nabla \ec_i(x) + \sum_{j \in \ca{F}(x)} \mu_j \nabla \ic_j(x) \Big).
			\end{aligned}
		\end{equation*}
		Combining with Lemma \ref{Le_Ja_nullspace}, it holds that 
		\begin{equation*}
			\mathrm{range}(\Jc(x)) \bigcap \Big( \partial f(x) + \sum_{i \in[N_E]} \lambda_i \nabla \ec_i(x) + \sum_{j \in \ca{F}(x)} \mu_j \nabla \ic_j(x) \Big) \neq \emptyset. 
		\end{equation*}
		Therefore, we can conclude that  there exists $\rho \in \Rp$ such that 
		\begin{equation*}
			0 \in	\partial f(x) + \sum_{l \in [p]} \rho_l c_l(x) + \sum_{i \in[N_E]} \lambda_i \nabla  u_i(x)  + \sum_{j \in \ca{F}(x)} \mu_j \nabla v_j(x),
		\end{equation*}
		which implies that $x$ is a KKT point to \ref{Prob_Ori}. This completes the proof. 
	\end{proof}

	\begin{lem}
		\label{Le_JA_JAIn}
		For any given $x \in \M$ and any $y \in \BOmegax{x}$, it holds that 
		\begin{equation*}
			\norm{\Ja(y) (\Ja(y) - I_n) } \leq \frac{2\La (2\Ma  + 1)}{\Lsc } \norm{c(y)}. 
		\end{equation*}
	\end{lem} 
	\begin{proof}
		For any $\y \in \BOmegax{x}$, it holds that
		\begin{equation*}
			\begin{aligned}
				&\norm{(\Ja(y) - I_n) \Ja(y)} 
				\leq \norm{\Ja(y)(\Ja(y) - \Ja(x))} + \norm{(\Ja(y) - \Ja(x))\Ja(x)} + 
				\norm{\Ja(y) - \Ja(x)}
				\\
				\leq{}& \La (2\Ma  + 1) \mathrm{dist}(y, \M)  \\
				\leq{}& \frac{2\La (2\Ma  + 1)}{\Lsc } \norm{c(y)}.
			\end{aligned}
		\end{equation*}
		Here the second inequality follows the Lipschitz continuity of $\Ja(\y)$ and  Lemma \ref{Le_Ja_idempotent}. The last inequality is directly from Lemma \ref{Le_bound_cx}. Then we complete the proof. 
	\end{proof}

	\paragraph{Proof for Theorem \ref{The_equivalence_FOSP}}
	\begin{proof}
		For any $\y \in \BOmegax{x}$, $\lambda \in \bb{R}^{N_E}$, and $\mu \in \bb{R}^{N_I}$, it holds from Lemma \ref{Le_JA_JAIn} that 
		\begin{align*}
			&\sup_{w \in \partial (f\circ \A)(x)}\norm{(\Ja(\y) - I_n)w} \leq \frac{2\La\Lf (2\Ma  + 1)}{\Lsc } \norm{c(y)},\\
			&\norm{(\Ja(\y) - I_n) \sum_{i \in [N_E]} \lambda_i \nabla (\ec_i\circ \A)(y)} \leq \frac{2\La( \norm{\lambda}_1\Mu ) (2\Ma  + 1)}{\Lsc } \norm{c(y)},\\
			&\norm{(\Ja(\y) - I_n) \sum_{j \in [N_I]} \mu_j \nabla (\ic_j\circ \A)(y)} \leq \frac{2\La( \norm{\mu}_1\Mv ) (2\Ma  + 1)}{\Lsc } \norm{c(y)}.
		\end{align*} 
		Here the last two inequalities follow from the fact that  $\norm{\nabla \ec(\A(y))} \leq \Mu$ and $\norm{\nabla \ic (\A(y))} \leq \Mv$. 
		
		Next,  we  estimate a lower bound of $\norm{(\Ja(\y) - I_n)\Jc(\y) c(\y)}$ in the followings,
		\begin{equation}
			\label{Eq_The_equivalence_FOSP_01}
			\begin{aligned}
				&\norm{(\Ja(\y) - I_n)\Jc(\y) c(\y)} \geq \norm{\Jc(\y)c(\y)} - \norm{\Ja(\y)\Jc(\y)c(\y)} \\
				\geq{}& \frac{\Lsc }{2}\norm{c(\y)} - \norm{\Ja(\y)\Jc(\y)} \norm{c(\y)} \geq \frac{\Lsc }{2}\norm{c(\y)} - \Lac \mathrm{dist}(y, \M) \norm{c(\y)} \\
				\geq{}& \frac{\Lsc }{2}\norm{c(\y)} - \Lac \varepsilon_x \norm{c(\y)} \geq  \frac{\Lsc }{4}\norm{c(\y)}.
			\end{aligned}
		\end{equation}
		
		Therefore, we can conclude that for any $\y \in \BOmegax{x}$, and any $w \in \partial_x \LCDP (y,\lambda,\mu)$, 
		\begin{equation*}
			\begin{aligned}
				&\norm{w} \geq \frac{1}{\Ma + 1}\norm{( \Ja(\y) - I_n ) w} \\
				\geq{}& \frac{1}{\Ma + 1}  \left(\Big(\beta + \sum_{i \in [N_E]} \lambda_i \tau_i + \sum_{j \in [N_I]} \mu_j \gamma_j \Big) \norm{(\Ja(\y) - I_n)\Jc(\y) c(\y)}\right) \\
				& -\frac{\norm{(\Ja(\y) - I_n) \partial (f\circ \A)(\y)}}{\Ma + 1} - \frac{\norm{(\Ja(\y) - I_n) \Ja(y)  \sum_{i \in [N_E]} \lambda_i \nabla (\ec_i\circ \A)(y)}}{\Ma + 1}\\
				& - \frac{\norm{(\Ja(\y) - I_n) \Ja(y)\sum_{j \in [N_I]} \mu_j \nabla (\ic_j\circ \A)(y)}}{\Ma + 1} \\
				\geq{}& \frac{1}{\Ma + 1}\left( \frac{ \Lsc }{4}\Big(\beta + \sum_{i \in [N_E]} \lambda_i \tau_i + \sum_{j \in [N_I]} \mu_j \gamma_j \Big) - \frac{2\La\Mxlambdamu (2\Ma  + 1)}{\Lsc } \right)\norm{c(\y)}\\
				\geq{}& \left( \frac{ \Lsc }{4(\Ma + 1)}\Big(\beta + \sum_{i \in [N_E]} \lambda_i \tau_i + \sum_{j \in [N_I]} \mu_j \gamma_j \Big) - \frac{4\La\Mxlambdamu}{\Lsc } \right)\norm{c(\y)}.
			\end{aligned}
		\end{equation*}
Since $w$ is arbitrary, the inequality \eqref{thm3.8-eq-1} in Theorem \ref{The_equivalence_FOSP} follows.
				
		Furthermore, when $\beta + \sum_{i \in [N_E]} \lambda_i \tau_i + \sum_{j \in [N_I]} \mu_j \gamma_j\geq \frac{32\La(\Ma  + 1)\Mxlambdamu }{\Lsc^2 }$, we get
		\begin{equation*}
		\mathrm{dist}(0, \partial_x \LCDP (y,\lambda,\mu) )
			\geq \frac{ \Lsc}{8(\Ma + 1)}\Big(\beta + \sum_{i \in [N_E]} \lambda_i \tau_i + \sum_{j \in [N_I]} \mu_j \gamma_j\Big) \norm{c(y)}. 
		\end{equation*}		
		Thus, for any KKT point
		 $y$ of \ref{Prob_Pen} that satisfies $y\in  \BOmegax{x}$, the left-hand side of the above inequality is $0$ and hence $\norm{c(y)} = 0$. As a result $y\in \K$.
		 Together with Theorem  \ref{The_equivalence_feasible} we conclude that $y$ is a KKT point of \ref{Prob_Ori}. This completes the proof. 
		
	\end{proof}

	\paragraph{Proof for Theorem \ref{The_fval_reduction}}
	\begin{proof}
		Firstly, it directly  follows from mean-value theorems, Lemma \ref{Le_Ax_c} and  Lemma \ref{Le_A_secondorder_descrease}  that
		\begin{equation*}
			\begin{aligned}
				&\left|\left(f(\A^{2}(\y)) + \lambda\tp u(\A^{2}(\y)) + \mu\tp v( \A^{2}(y) )\right) - \left(f(\A(\y)) + \lambda\tp u(\A(\y)) + \mu\tp v( \A(y) )\right) \right| \\
				\leq{}& \Mxlambdamu  \norm{\A^{2}(\y) - \A(\y)} \leq \frac{2\Mxlambdamu (\Ma +1)}{\Lsc }\norm{c(\A(\y))}
				\leq \frac{8\Mxlambdamu (\Ma +1)\Lac }{\Lsc ^3} \norm{c(\y)}^2.
			\end{aligned}
		\end{equation*}		
		Recall Lemma \ref{Le_A_secondorder_descrease}, we know that the inequality $\norm{c(\A(y))} \leq \frac{1}{2} \norm{c(y)}$ holds for any $y \in \Omegax{x}$. As a result, 
		\begin{equation*}
			\begin{aligned}
				&\LCDP( \A(y), \lambda, \mu ) - \LCDP( y, \lambda, \mu ) \\
				\leq{}& \left|\left(f(\A^{2}(\y)) + \lambda\tp u(\A^{2}(\y)) + \mu\tp v( \A^{2}(y) )\right) - \left(f(\A(\y)) + \lambda\tp u(\A(\y)) + \mu\tp v( \A(y) )\right) \right| \\
				&+ \frac{1}{2}\Big(\beta + \sum_{i \in [N_E]} \lambda_i \tau_i + \sum_{j \in [N_I]} \mu_j \gamma_j\Big) \left( \norm{c(\A(\y))}^2 - \norm{c(\y)}^2 \right)\\
				\leq{}& - \left(\frac{1}{4}\Big(\beta + \sum_{i \in [N_E]} \lambda_i \tau_i + \sum_{j \in [N_I]} \mu_j \gamma_j\Big) -\frac{8\Mxlambdamu (\Ma +1)\Lac }{\Lsc ^3}  \right)\norm{c(\y)}^2 \leq 0.
			\end{aligned}
		\end{equation*}
		Here the last inequality uses the fact that 
		\begin{equation*}
			\beta + \sum_{i \in [N_E]} \lambda_i \tau_i + \sum_{j \in [N_I]} \mu_j \gamma_j \geq \frac{8\Mxlambdamu (\Ma +1)\Lac }{\Lsc ^3}.
		\end{equation*} 
		Then we complete the proof. 
	\end{proof}

	\subsection{Proof for Section 3.3}
	\label{Subsection_proofs_ineqs}

	\begin{lem}
		\label{Le_JAinf}
		For any $x \in \M$, it holds that 
		\begin{equation*}
			J_{\A^{\infty}}(x) =  \Ja(x).
		\end{equation*}
	\end{lem}
	\begin{proof}
		Lemma \ref{Le_Ja_idempotent} illustrates that $\Ja(x)^2 = \Ja(x)$ holds for any $x \in \M$. Therefore, for any $k \geq 1$, it holds that 
		\begin{equation*}
			J_{A^k}(x)  = J_{A^{k-1}}(x)\Ja(A^{k-1}(x))  =   J_{A^{k-1}}(x) \Ja(x) = J_{A^{k-2}}(x) \Ja(x)^2 = J_{A^{k-2}}(x) \Ja(x) = J_{A^{k-1}}(x). 
		\end{equation*}
		Therefore, we can conclude that $J_{A^k}(x) = \Ja(x)$ holds for any $k \geq 1$. As a result, we obtain that 
		\begin{equation*}
			J_{\A^{\infty}}(x) =  \Ja(x),
		\end{equation*}
		and this completes the proof. 
	\end{proof}
	
	The following proposition is a direct corollary from \cite{xiao2022dissolving}, hence we omit its proof for simplicity. 
	\begin{prop}	
		\label{Prop_postprocessing}
		For any given $x \in \M$, suppose $\beta$ and $\{\gamma_j\}$ satisfy Condition \ref{Cond_beta_gamma}, then the following inequalities hold for any $\y \in \BOmegax{x}$
		\begin{align}
			&h(\A^{\infty}(\y)) - h(\y)\leq  - \frac{\beta}{4}\norm{c(\y)}^2\\
			& v_i(\A^{\infty}(\y)) \leq{} v_i(\y) + \frac{\gamma_i}{4}\norm{c(\y)}^2. 
		\end{align}
	\end{prop}

	\paragraph{Proof for Lemma \ref{Le_relationship_ACQ_ineq}}
	\begin{proof}
		From the expression of \ref{Prob_Ori} and \ref{Prob_Pen}, it holds that $\K \subseteq \KA$, hence
		\begin{equation*}
			\TK(x) \subseteq \TKA(x),\quad \text{and} \quad  \TKlin(x) \subseteq \TKAlin(x).
		\end{equation*}
		Then together with Lemma \ref{Le_relationship_ACQ}, we obtain 
		\begin{equation*}
			\TK(x) \subseteq \Ja(x)\tp \TKA(x), \quad \text{and} \quad \TKlin(x) = \Ja(x)\tp \TKAlin(x). 
		\end{equation*}

		Furthermore, for any $d \in \TKA(x)$, there exists a sequence $\{\xk\} \subset \KA \cap \BOmegax{x}$ converging to $x$ and a sequence $\{t_k\} \subset \bb{R}_+$ converging to $0$, such that $\lim\limits_{k \to +\infty} \frac{\xk - x}{t_k} = d$. 
		Then from Proposition \ref{Prop_postprocessing}, for any $\xk$ and any $j \in [N_I]$, it holds that
		\begin{equation*}
			\begin{aligned}
				v_j(\A^{\infty}(\xk))  = v_j(\A^{\infty}(\A(\xk)))
			 \leq{}& v_j(\A(\xk)) + \frac{\gamma_i}{4} \norm{c(\A(\xk))}^2
			\\
			\leq{}& v_j(\A(\xk)) + \frac{\gamma_i}{16} \norm{c(\xk)}^2 
			\leq \tilde{v}_j(\xk)\leq 0,
			\end{aligned}
		\end{equation*}
		which implies that $\{\A^{\infty}(\xk)\} \subset \K$. Note that we used the
		fact that $\norm{c(\A(\xk))} \leq \frac{1}{2}\norm{c(\xk)}$ from Lemma \ref{Le_A_secondorder_descrease} in the second inequality.
		
		Therefore, it holds from Lemma \ref{Le_JAinf} that
		\begin{equation*}
			\TK(x) \ni \lim\limits_{k \to +\infty} \frac{\A^{\infty}(\xk) - \A^{\infty}(x)}{t_k} = J_{\A^{\infty}}(x)\tp d = \Ja(x)\tp d,
		\end{equation*}
		From the arbitrariness of $d \in \TKA(x)$, it holds that $\Ja(x)\tp \TKA(x) \subseteq \TK(x)$. Then together with Lemma \ref{Le_relationship_ACQ}, we obtain
		\begin{equation*}
			\TK(x) = \Ja(x)\tp \TKA(x).
		\end{equation*}
		Furthermore,  directly from Lemma \ref{Le_subspace_polar_cone},  we can conclude that 
		\begin{equation*}
			\TKA(x)^\circ \cap \mathrm{range}(\Ja(x)) = \Ja(x)\TK(x)^\circ,
		\end{equation*}
		and complete the proof. 
	\end{proof}

	\paragraph{Proof for Theorem \ref{The_equivalence_feasible_Ineq}}	
	\begin{proof}
		It holds straightforwardly from Theorem \ref{The_equivalence_feasible} that any first-order stationary point of \ref{Prob_Pen} is a first-order stationary point of \ref{Prob_Ori}. 
		
		On the other hand, suppose $x$ is a first-order stationary point of \ref{Prob_Ori}, then it holds that 
		\begin{equation*}
			0 \in \partial f(x) + \TK(x)^\circ.
		\end{equation*}
		As a result, it holds from Lemma \ref{Le_relationship_ACQ_ineq} that 
		\begin{equation*}
			\begin{aligned}
				&0 \in \Ja(x)\left( \partial f(x) + \TK(x)^\circ \right)\\
				={}& \Ja(x) \partial f(x) + \TKA(x)^\circ \cap \mathrm{range}(\Ja(x)) \\
				={}& ( \partial h(x) +  \TKA(x)^\circ) \cap \mathrm{range}(\Ja(x)) \\
				\subseteq{}& \partial h(x) + \TKA(x)^\circ. 
			\end{aligned}
		\end{equation*}
		Here the second equality used the fact that $c(x) = 0$. Therefore, it holds that $x$ is a first-order stationary point of \ref{Prob_Pen}. As a result, we obtain that $x \in \K$ is a first-order stationary point of \ref{Prob_Ori} if and only if $x$ is a first-order stationary point of \ref{Prob_Pen}. This completes the proof.
	\end{proof}

	\paragraph{Proof for Theorem \ref{The_equivalence_Ineq_FOSP}}
	\begin{proof}
	We prove this theorem by contradiction. Suppose there exists a $y \in \BOmegax{x}$ that is a first-order stationary point of \ref{Prob_Pen} with $c(y) \neq 0$. From Definition \ref{Defin_FOSP}, we first obtain that $y \in \KA$. Then we aim to find a direction $\tilde{d} \in \TKA(y)$ that produces a sufficient decrease for $h$. 
		
		 Let $\tilde{d}:= (\Ja(y)\tp - I_n)\Jc(y) c(y)$, we first prove that $\tilde{d} \in \TKA(y)$. For any $j \in [N_I]$, it holds from Lemma \ref{Le_JA_JAIn} that 
		\begin{equation}
			\begin{aligned}
				&\inner{\tilde{d}, \nabla (\ic_j\circ \A)(y)} \\
				={}& \inner{\tilde{d}, \Ja(y) \nabla \ic_j(\A(y))} =  \inner{ \Jc(y) c(y), (\Ja(y) - I_n) \Ja(y) \nabla \ic_j(\A(y))}\\
				\leq{}& \frac{2\La (2\Ma  + 1)\Mv}{\Lsc } \norm{c(y)} \norm{\Jc(y) c(y)}\leq  \frac{4\La (2\Ma  + 1)\Mv}{\Lsc^2 }\norm{\Jc(y) c(y)}^2.
			\end{aligned}
		\end{equation}
	
		On the other hand, it follows \eqref{Eq_The_equivalence_FOSP_01} that 
		\begin{equation}
			\label{Eq_The_equivalence_Ineq_FOSP_0}
			\begin{aligned}
				&\inner{(\Ja(y)\tp - I_n)\Jc(y) c(y), \Jc(y) c(y)} = -\norm{ \Jc(y) c(y)}^2 + \inner{\Jc(y) c(y), \Ja(y)\Jc(y) c(y)}\\
				\leq{}& -\norm{ \Jc(y) c(y)}^2 + \norm{\Ja(y)\Jc(y)} \norm{c(y)} \norm{\Jc(y) c(y)} \\
				\leq{}& -\norm{ \Jc(y) c(y)}^2 + \frac{\varepsilon_x \Lac}{\Lsc}  \norm{\Jc(y) c(y)}^2 \leq -\frac{1}{2} \norm{ \Jc(y) c(y)}^2.
			\end{aligned}
		\end{equation}
		Then it holds that 
		\begin{equation}
			\label{Eq_The_equivalence_Ineq_FOSP_1}
			\begin{aligned}
				&\inner{\tilde{d}, \nabla \tilde{\ic}_j(y)} = \inner{\tilde{d}, \nabla (\ic_j\circ \A)(y)} + \gamma_i \inner{(\Ja(y)\tp - I_n)\Jc(y) c(y), \Jc(y) c(y)}\\
				\leq{}& \left(\frac{4\La (2\Ma  + 1)\Mv}{\Lsc^2 } - \frac{\gamma_i}{2}\right) \norm{\Jc(y) c(y)}^2 < 0. 
			\end{aligned}
		\end{equation}
		Here the last inequality follows from 
		Condition \ref{Cond_beta_gamma}. 
		As a result, we can conclude that $\inner{\tilde{d}, \nabla \tilde{\ic}_j(y)} < 0$ holds for any $1 \leq j \leq N_I$. Then from the Lipschitz smoothness of $\tilde{v}_j$, there exists a constant $\hat{\varepsilon}_0>0$ such that $\tilde{\ic}(y + t \tilde{d}) < \tilde{\ic}(y)$ holds for any $t\in [0, \hat{\varepsilon}_0]$. 
		Together with the fact that $y \in \KA$, we can conclude that  $y+t\tilde{d} \in \KA$ holds for  any $t\in [0, \hat{\varepsilon}_0]$, hence $\tilde{d} \in \TKA(y)$. 
		
		Furthermore, since $0 \in \partial h(y) + \TKA(y)^\circ$,  there exists $w \in \partial f(\A(y))$ such that $-\big(\Ja(y) w + \beta \Jc(y)c(y)\big) \in \TKA(x)^\circ$.  
		Since $\tilde{d} \in \TKA(y),$ then it holds that 
		\begin{equation}
			\label{Eq_The_equivalence_Ineq_FOSP_2}
			\begin{aligned}
				0\leq{}&  \inner{ \tilde{d}, \Ja(y) w + \beta \Jc(y)c(y) } \\
				\overset{(i)}{\leq}{}& \frac{2\La \Lf (2\Ma  + 1)}{\Lsc } \norm{c(y)} \norm{\Jc(y)c(y)} - \frac{\beta}{2} \norm{\Jc(y)c(y)}^2 \\
				\leq{}& \left(\frac{2\La \Lf (2\Ma  + 1)}{\Lsc^2 }- \frac{\beta}{2}\right)\norm{\Jc(y)c(y)}^2\\
				\overset{(ii)}{\leq}{}&  -\frac{\beta}{4}  \norm{ \Jc(y) c(y)}^2 < 0,
			\end{aligned}
		\end{equation}
		Here $(i)$ follows from Lemma \ref{Le_JA_JAIn} and
		 \eqref{Eq_The_equivalence_Ineq_FOSP_0}. Moreover,  $(ii)$ follows from Condition \ref{Cond_beta_gamma}, which illustrates that 
		\begin{equation}
			\beta \geq \frac{64\Lf(\Ma +1)(\Lac + \Lsc \La) }{\Lsc ^3} \geq \frac{64\Lf(\Ma +1) \La }{\Lsc ^2} \geq \frac{8\La \Lf (2\Ma  + 1)}{\Lsc^2 }.
		\end{equation}

		  As a result, from \eqref{Eq_The_equivalence_Ineq_FOSP_2} we get the contradiction and achieve that $c(y) = 0$. Thus $y\in\ca{K}$. 
		  Therefore, from Theorem \ref{The_equivalence_feasible_Ineq} we conclude that $y$ is a first-order stationary point of \ref{Prob_Ori}. 
	\end{proof}

	\paragraph{Proof for Corollary \ref{Coro_equivalence_Ineq_KKT}}
	\begin{proof}
		For any $y \in \BOmegax{x}$ that is a first-order stationary point of \ref{Prob_Pen}, let $\tilde{d}:= (\Ja(y)\tp - I_n)\Jc(y) c(y)$, then from \eqref{Eq_The_equivalence_Ineq_FOSP_1} we can conclude that $\inner{\tilde{d}, \nabla \tilde{\ic}_j(y)} \leq 0$ holds for any $j \in [N_I]$. Thus we get that $\tilde{d} \in \TKAlin(y)$. 
		
		Note that $y$ is a KKT point of \ref{Prob_Pen}, then there exists $\tilde{\mu} \in \bb{R}^{N_I}_+$ and $w \in \partial f(\A(y))$ such that  $0 = \Ja(y)w + \beta \Jc(y)c(y) + \sum_{j \in [N_I]} \tilde{\mu}_j \nabla \tilde{\ic}_j(y)$. 
		Therefore, similar to \eqref{Eq_The_equivalence_Ineq_FOSP_2},  we have
		\begin{equation}
			\begin{aligned}
				0 ={}&\inner{\tilde{d}, \Ja(y)w + \beta \Jc(y)c(y) + \sum_{j \in [N_I]} \tilde{\mu}_j \nabla \tilde{\ic}_j(y)}\\
				\leq{}& \inner{\tilde{d}, \Ja(y)w + \beta \Jc(y)c(y)}\\
				\leq{}& \left(\frac{2\La \Lf (2\Ma  + 1)}{\Lsc^2 }- \frac{\beta}{2}\right)\norm{\Jc(y)c(y)}^2\\
				\leq{}& -\frac{\beta}{4}  \norm{ \Jc(y) c(y)}^2 \leq 0. 
			\end{aligned}
		\end{equation}
		which illustrates that $c(y) = 0,$ and hence $y\in \ca{K}$. Therefore, from Theorem \ref{The_equivalence_feasible} we can conclude that $y$ is a KKT point to \ref{Prob_Ori}. On the other hand, it directly holds from Theorem \ref{The_equivalence_feasible} that any KKT point to \ref{Prob_Ori} is a KKT point to \ref{Prob_Pen}. This completes the proof. 
	\end{proof}

\end{document}